\documentclass{amsart}
\usepackage{amsthm,amstext,amsmath,amscd,amssymb,latexsym, mathrsfs}
\usepackage{color}
\usepackage{graphicx}
\usepackage[T1]{fontenc}
\usepackage{txfonts}

\usepackage[matrix,arrow]{xy}

\usepackage{amsfonts,enumerate,array,hyperref}

\newcommand{\C}{{\mathbb C}}

\newcommand{\K}{{\mathbb K}}

\newcommand{\N}{{\mathbb N}}
\renewcommand{\P}{{\mathbb P}}
\newcommand{\Q}{{\mathbb Q}}

\DeclareMathOperator{\lcm}{LCM}

\newcommand{\cO}{{\mathcal O}}

\newcommand{\punkt}{\HHspace{-.3ex}\raise.15ex\HHbox to1ex{\HHuge.}}

\DeclareMathOperator{\rank}{rank}

\newcommand{\CC}{\C}

\newcommand{\QQ}{\Q}

\newcommand{\paper}{: \begin{it}}
\newcommand{\jour }{, \end{it}}

\newtheorem{theorem}{Theorem}[section]
\newtheorem{lemma}[theorem]{Lemma}
\newtheorem{proposition}[theorem]{Proposition}
\newtheorem{corollary}[theorem]{Corollary}

\theoremstyle{definition}
\newtheorem{definition}[theorem]{Definition}
\newtheorem{notation}[theorem]{Notation}
\newtheorem{example}[theorem]{Example}

\newtheorem{remark}[theorem]{Remark}

\numberwithin{equation}{section}

\begin{document}
\title[Eigenschemes and the Jordan canonical form]{Eigenschemes and the Jordan canonical form}
\author[H. Abo]{Hirotachi Abo}
	\address{Department of
             Mathematics, University of Idaho, Moscow, ID 83884, USA}
          \email{abo@uidaho.edu}
     	\urladdr{http://www.webpages.uidaho.edu/$\sim$abo}
\author[D. Eklund]{David Eklund}
	\address{Prover Technology AB, Rosenlundsgatan 54, 118 63 Stockholm, Sweden} 
	\email{daek@kth.se} 
\author[T. Kahle]{Thomas Kahle}
	\address{Fakult\"at f\"ur Mathematik, OvGU Magdeburg,  39106 Magdeburg, Germany}
         \email{thomas.kahle@ovgu.de}
         \urladdr{http://www.thomas-kahle.de/}
\author[C. Peterson]{Chris Peterson}
    \address{Department of
             Mathematics, Colorado State University, Fort Collins, CO 80523, USA}
    \email{peterson@math.colostate.edu}
    \urladdr{http://www.math.colostate.edu/$\sim$peterson}
\begin{abstract}
  We study the \emph{eigenscheme} of a matrix which encodes
  information about the eigenvectors and generalized eigenvectors of a
  square matrix. The two main results in this paper are
  a decomposition of the eigenscheme of a matrix into primary
  components and the fact that 
  this decomposition encodes the numeric data of the Jordan canonical
  form of the matrix.  We also describe how the eigenscheme can be
  interpreted as the zero locus of a global section of the tangent
  bundle on projective space. This interpretation allows one to see
  eigenvectors and generalized eigenvectors of matrices from an
  alternative viewpoint.
\end{abstract}

\maketitle

\setcounter{tocdepth}{1}
\tableofcontents

\section{Introduction}
\subsection*{Motivation}
Let $\K$ be a field about which we make no a priori assumptions.
For any $r\times r$ matrix $A \in \K^{r\times r}$, a non-zero vector $\boldsymbol v \in \K^r$ is an {\it eigenvector} if $A \boldsymbol v$ is a scalar multiple of $\boldsymbol v$ or equivalently if $A \boldsymbol v$ and $\boldsymbol v$ are linearly dependent. The span of $\boldsymbol v$ is a one dimensional subspace of $\K^r$ and its projectivization can thus be viewed as a point in the projectivization of $\K^r$. In other words, the  subspace of $\K^r$ spanned by an eigenvector $\boldsymbol v$ of $A$ determines a point $[\boldsymbol v] \in \P^{r-1}$.  Let $R = \K[x_1, \dots, x_r]$, let $\boldsymbol x$ be the column vector of variables, and let $(A\boldsymbol x \ | \ \boldsymbol x)$ be the $r \times 2$ matrix obtained by concatenating the two $r \times 1$ matrices $A\boldsymbol x$ and $\boldsymbol x$ side by side. Then a non-zero vector $\boldsymbol v \in \K^r$ is an eigenvector of $A$ if and only if the matrix obtained from $(A\boldsymbol x \ | \ \boldsymbol x)$ by evaluating at $\boldsymbol v$ has rank one. As a consequence, $\boldsymbol v$ is an eigenvector of $A$ if and only if it is in the common zero locus of the set of all $2 \times 2$ minors of~$(A\boldsymbol x \ | \ \boldsymbol x)$. Let $I_A$ denote the ideal generated by the $2 \times 2$ minors of~$(A\boldsymbol x \ | \ \boldsymbol x)$. The set of all projectivized eigenvectors of $A$, denoted $E_A$, is the algebraic subset of $\P^{r-1}$ defined by the homogeneous ideal $I_A$.

The purpose of this paper is to study the {\it scheme} $Z_A$ of~$I_A$. We call $Z_A$ the {\it eigenscheme} of $A$. If $A$ is not diagonalizable, then the geometry of $Z_A$ is richer than the geometry of~$E_A$. Note that $E_A$ is the union of linear subvarieties of $\P^{r-1}$ (corresponding to the projectivization of the eigenspaces of~$A$) while $Z_A$ is a scheme supported on $E_A$. We show that the non-reduced structure of $Z_A$ carries information about the generalized eigenvectors of $A$ encoded by the nilpotent Jordan structure.  We also show that the vector of $2 \times 2$ minors of~$(A\boldsymbol x \ | \ \boldsymbol x)$ can be identified with a global section of the tangent bundle on $\P^{r-1}$ and that $Z_A$ can be identified with the zero scheme of this section.
Some of the properties of the eigenscheme can therefore be derived from properties of the tangent bundle and certain problems concerning eigenvectors of matrices can be naturally translated into problems concerning the tangent bundle on $\P^{r-1}$.

\subsection*{Results}
To study $Z_A$, we first find a primary decomposition of $I_A$. Each primary component of $I_A$ is associated to a particular eigenvalue of $A$.   If the characteristic plynomial of $A \in \K^{r \times r}$ splits in $\K$, then $A$ is similar to a matrix $J$ in Jordan canonical form.  Since $J$ is unique up to a permutation of the Jordan blocks, it can be reconstructed if one knows the following two things for each eigenvalue $\lambda$ of $A$; namely, (i) the number of different Jordan blocks with eigenvalue $\lambda$ and (ii) the size of each Jordan block. We show that the primary decomposition of $I_A$  contains sufficient information to deduce each of (i), (ii). Indeed, a primary component of $I_A$ is associated to the collection of all Jordan blocks of $A$ that are of the same size and have the same eigenvalue. Furthermore, the dimension of the primary component corresponds to the number of Jordan blocks of the given size and its degree corresponds to the size of each Jordan block.  Our strategy for finding a primary decomposition of~$I_A$ is as follows.

Let $A,B \in \K^{r\times r}$. Suppose that $A$ is similar to $B$, i.e., there exists an invertible $C \in \K^{r\times r}$ such that $A = C^{-1}BC$. Then $I_A$ and $I_B$ differ only by the linear change of variables determined by $C$. Equivalently the eigenschemes of $A$ and $B$ differ only by the automorphism of $\P^{r-1}$ induced by $C$ (see Proposition~\ref{prop:similar} for more details). This means that from a primary decomposition of $I_B$, one can obtain a primary decomposition of $I_A$ by applying the linear change of variables determined by $C$. Thus, without significant loss of generality, we may assume that $A$ is a Jordan matrix~$J$, at the expense of possibly a finite extension of~$\K$.

The key idea for finding a primary decomposition of $I_J$ is to
decompose $I_J$ into ideals, each of which is paired with a different
eigenvalue of~$J$.
This decomposition allows us to reduce the problem of finding a
primary decomposition of the ideal of a Jordan matrix to the problem
of finding a primary decomposition of the ideal of a
Jordan matrix with a single eigenvalue. The ideal $I_J$ associated
with a Jordan matrix $J$ with a single eigenvalue is a \emph{binomial
ideal}, i.e., an ideal generated by {\it binomials}, polynomials with
at most two terms. We use the general theory of binomial ideals
developed by D.~Eisenbud and B.~Sturmfels in~\cite{eisenbud-sturmfels}
to construct a primary decomposition of $I_J$. In the following two
paragraphs, we illustrate this idea in more detail.

First suppose that $J$ has two or more distinct eigenvalues
$\lambda_1,\dots,\lambda_n \in K$.  After a suitable permutation of
Jordan blocks, $J$ can be written as a block diagonal matrix whose
main diagonal matrices are the Jordan matrices
$J_{\lambda_1}, \dots, J_{\lambda_n}$ with single
eigenvalues~$\lambda_1, \dots, \lambda_n$.
For each $i \in \{1, \dots, n\}$, let $r_i$ be the size of
$J_{\lambda_i}$ and let $\boldsymbol x^{(i)}$ be the column vector of
the variables $x_1^{(i)}, \dots, x_{r_i}^{(i)}$. Consider the
polynomial
ring~$R = \K\left[\left. x_1^{(i)}, \cdots, x_{r_i}^{(i)} \ \right| \
  1 \leq i \leq n \right]$.
Theorem~\ref{thm:decomposition} shows that $I_J$ can be written as the
intersection of the ideals, each of which is generated by the
$2 \times 2$ minors of the $r \times 2$ matrix
$\left(\left. J_{\lambda_i}\boldsymbol x^{(i)} \ \right| \ \boldsymbol
  x^{(i)}\right)$
and the variables that are not paired with~$J_{\lambda_i}$. This
indicates that, to find a primary decomposition of the ideal of a
general Jordan matrix, it is sufficient to find a primary
decomposition of the ideal of a Jordan matrix with a single
eigenvalue.
 
If $J$ has a single eigenvalue, then $I_J$ is a binomial ideal. In~\cite{eisenbud-sturmfels}, D.~Eisenbud and B.~Sturmfels showed that, over an algebraically closed field, every binomial ideal has a primary decomposition into primary binomial ideals.  They also provided algorithms for finding such a primary decomposition.  Several improvements of the decomposition theory and the algorithm have been implemented (see, for example, \cite{castilla-sanchez,kahle2010, kahle2012, KM2014}). The first step of a binomial ideal decomposition is to decompose the ideal into \emph{cellular} binomial ideals, modulo which every variable is either a non-zerodivisor or nilpotent.  In Proposition~\ref{prop:decomposition} we give a cellular decomposition which turns out to be a primary decomposition too.  Cellular decomposition is field independent and thus the field assumptions from binomial primary decomposition do not conern us much.  When talking about a given matrix $A$, though, we must often assume that $\K$ contains the eigenvalues of~$A$.

\subsection*{Jordan canonical forms in commutative algebra}
The study of commutative algebra aspects 
of the Jordan canonical form can be extended to
(1) the Kronecker--Weierstrass theory of matrix pencils 
and (2) the theory of eigenvectors of tensors.  
We now briefly discuss commutative 
algebra perspectives on~(1) and~(2).  
For the results we mention in this subsection, 
assumptions on the field $\K$ may be necessary.  An algebraically
closed field of characteristic zero is sufficient in any case, but
weaker assumptions often suffice.  
The interested reader should consult the specific references in each
case.

(1) Let $R = \K[x_1,\dots, x_r]$, let $M$ be an $s \times 2$ matrix of
linear forms from $R$, and let~$I_M$ be the ideal generated by the
$2 \times 2$ minors of $M$. The height of $I_M$ is known to be less
than or equal to $s-1$.  If equality holds, then the Eagon-Northcott
complex of $M^T$ (considered as a graded homomorphism between two free
graded $R$-modules) is a minimal free resolution of $R/I_M$.
M.~L.~Catalano-Johnson studied the minimal free resolution of $I_M$ in
some cases where $I_M$ does not have the expected
codimension~\cite{catalano-johnson} as follows.

Let $A, B \in \K^{s\times r}$ such that
$M = (A \boldsymbol x \ | \ B \boldsymbol x)$.  The
Kronecker-Weierstrass normal form of the pencil of $A$ and $B$ is used
to transform $M$ to Kronecker-Weierstrass form which is
another~$s \times 2$ matrix~$KW(M)$ of linear forms.  The matrix
$KW(M)$ is a concatenation of ``scroll blocks,'' ``nilpotent blocks,''
and/or ``Jordan blocks'' (see~\cite{catalano-johnson} for the
definitions of these different types of blocks). Because of the
way~$KW(M)$ is constructed, the ideals~$I_M$ and $I_{KW(M)}$ differ
only by a linear change of variables.  M.~L.~Catalano-Johnson used
$KW(M)$ to study homological aspects of $R/I_M$. This work includes an
explicit construction of the minimal free resolution of $R/I_M$
when~$KW(M)$ simultaneously has at least one Jordan block with
eigenvalue $0$ and no nilpotent blocks. This result was extended to
the general case by R.~Zaare-Nahandi and
R.~Zaare-Nahandi~\cite{zaare-nahandi} in~2001.

Recently, H.~D.~Nguyen, P.~D.~Thieu, and T.~Vu~\cite{nguyen-thieu-vu}
showed that $R/I_M$ is Koszul if and only if the largest length of a
nilpotent block of $KW(M)$ is at most twice the smallest length of a
scroll block. This result settled a conjecture of A.~Conca.

The algebraic set defined by $I_M$ consists of points of $\P^{r-1}$, each
of which is the equivalence class of a generalized eigenvector of $A$
and~$B$.  A possible extension of our original question is ``What is a
decomposition of the scheme of $I_M$ into primary
components?''  We have recently learned that work on this question is
under way by H.~D.~Nguyen and M.~Varbaro who kindly informed us about
their progress.

(2) The concept of eigenvectors of matrices was recently generalized to tensors by L.-H.~Lim~\cite{lim} and L.~Qi~\cite{qi} independently, and algebro-geometric aspects of tensor eigenvectors were studied by several authors (see, for example, ~\cite{Rob, ORS, ASS}). 

The eigenscheme of a tensor can be defined analogously. Another
possible extension of our original question is ``What is the primary
decomposition of the ideal of such a scheme?'' The decomposition of
the scheme of so-called orthogonally decomposable tensors into primary
components was described by E.~Robeva~\cite{Rob}.  However, the
primary decomposition of the ideal of the eigenscheme of a tensor is
not yet well understood.

\subsection*{Organization} In Section~\ref{sec:eigenschemes} we
introduce notation that is used throughout the paper, define the
eigenscheme of a matrix, and discuss a few examples of eigenschemes.
We also show that if two $r \times r$ matrices are similar, then the
corresponding eigenschemes differ only by an automorphism
of~$\P^{r-1}$.  The goal of Section~\ref{sec:jordan-single} is to find
a primary decomposition of the ideal $I_J$ of a Jordan matrix $J$ with
a single eigenvalue.  As stated before, the first step is to construct
a cellular decomposition $\bigcap_{i=1}^\ell I_i$ of~$I_J$. We then
show that $I_i$ is primary for each $i \in \{1, \dots, \ell\}$.  We
also compute the reduced Gr\"obner basis for each $I_i$ which enables
us to describe the Hilbert polynomial of the quotient ring
modulo~$I_i$.  In Section~\ref{sec:jordan}, we construct a primary
decomposition of the ideal of a general Jordan matrix. We use the
primary decomposition of such an ideal to prove that, assuming all
eigenvalues lie in~$\K$, a square matrix is diagonalizable if and only
if the ideal of the matrix is radical.  In
Section~\ref{sec:tangent-bundle}, we show how the eigenscheme of an
$r \times r$ matrix and the zero scheme of a global section of the
tangent bundle on $\P^{r-1}$ are related.  The characterization of
diagonalizable matrices allows a characterization of the hypersurface
formed by non-diagonalizable matrices, which we call the discriminant
hypersurface.  In Section~\ref{sec:discriminant}, we show that the
degree of the discriminant hypersurface can be expressed as a function
of a Chern class of the tangent bundle.

\subsection*{Acknowledgements} The first author would like to thank
the Simons Institute for the Theory of Computing at the University of
California at Berkeley, where much of this work was done. The third
author is supported by the research focus dynamical systems of the
state Saxony-Anhalt. The fourth author acknowledges support from
NSF~1228308, NSF~1322508.  The authors thank Ezra Miller for his 
valuable comments and suggestions to improve the quality of the paper.
\section{Eigenvectors and eigenschemes} 
\label{sec:eigenschemes}
In this section we define an eigenscheme of a square matrix and
discuss several examples of eigenschemes.  We begin this section by
introducing notations we use through the paper.

Let $A \in \K^{r\times s}$.  For each $i \in \{1,\dots,r\}$ and for each $j \in \{1, \dots, s\}$, we write $A[i,j]$ for the $(i,j)$-entry of $A$.  For a given $i \in \{1, \dots, r\}$,  $A[i,:]$ denotes the $i$-th row of $A$.
If $A \in \K^{r\times s}$ and $B \in \K^{r'\times s'}$ then $A \oplus B$ denotes the ($r+r')\times (s+s')$  {\it direct sum matrix} of $A$ and $B$.

Suppose that $A$ and $B$ have the same number of rows, i.e., $r=r'$, then we write $(A | B)$ for the matrix obtained by concatenating $A$ and $B$ side by side. 

%
Let $R$ be the graded polynomial ring $\K[x_1, \cdots, x_r] $ with standard grading,
and $\boldsymbol x$ be the column vector of its variables.  For any matrix $A \in  \K^{r\times r}$ let $L_A$ be the set of $2 \times 2$ minors of the augmented matrix $(A \boldsymbol x \ | \ \boldsymbol x)$.
Let $I_A$ be the homogeneous ideal generated by $L_A$. Then $\boldsymbol v \in \K^r \setminus \{\mathbf{0}\}$ is an eigenvector of $A$ if and only if the equivalence class $[\boldsymbol v] \in \P^{r-1} = \P(\K^r)$ containing $\boldsymbol v$ lies in the algebraic set $V(I_A)$ defined by~$I_A$. 
\begin{definition}
\label{def:eigenscheme}
For a given $A \in \K^{r \times r}$, the closed subscheme of
$\P^{r-1}$ associated to $I_A$ is the {\it eigenscheme} of~$A$.
\end{definition}
\begin{definition}
\label{def:linearspan}
The {\it scheme-theoretic linear span} of a subscheme $Z \subseteq
\P^{r-1}$ is the smallest linear subspace $L \subseteq \P^{r-1}$ such
that $Z$ is a subscheme of $L$.
\end{definition}
\begin{example}
Let
$A = \left(\begin{smallmatrix}
4 & 0 & 1 \\
2 & 3 & 2 \\
1 & 0 & 4
\end{smallmatrix}
\right) \in\QQ^{3\times 3}$.
Then $I_A$ is generated by the three quadrics
\begin{eqnarray*}
Q_0 & = & (4x_1+x_3)x_2-(2x_1+3x_2+2x_3)x_1,  \\
Q_1 & = & (4x_1+x_3)x_3-(x_1+4x_3)x_1, \\
Q_2 & = & (2x_1+3x_2+2x_3)x_3-(x_1+4x_3)x_2.  
\end{eqnarray*}
The decomposition of $I_A$ into primary ideals is 
\[
I_A = \langle x_1+x_3 \rangle \cap \langle x_2-2x_3, \ x_1-x_3\rangle. 
\]
Both $\langle x_1+x_3 \rangle$ and $\langle x_2-2x_3, \ x_1-x_3\rangle$ are minimal associated primes of $\QQ[x_1,x_2,x_3]/I_A$, and hence $V(I_A)$ consists of the projective line in $\P^2$ defined by $x_1+x_3=0$ and the point in $\P^2$ defined by the intersection of the two projective lines $x_2-2x_3=0$ and $x_1-x_3=0$.

It is straightforward to see that  the affine cones over $V(x_1+x_3)$ and $V(x_2-2x_3, \ x_1-x_3)$ are the eigenspaces corresponding to the eigenvalues $3$ and $5$ respectively. In particular, since $V(I_A)$ contains three linearly independent eigenvectors, the matrix $A$ is diagonalizable. 
\end{example}
\begin{example}
Let
$A = \left(\begin{smallmatrix}
2 & 1 & 1 \\
0 & 1 & 1 \\
0 & 0 & 1
\end{smallmatrix}
\right) \in \QQ^{3\times 3}$.
Its eigenvalues are $2$ and~$1$ and $I_A$ is generated by
\begin{eqnarray*}
Q_0 & = &  x_1x_2+x_2^2-x_1x_3+x_2x_3 \\
Q_1 & = & x_1x_3+x_2x_3+x_3^2 \\
Q_2 & = & x_3^2. 
\end{eqnarray*}
The primary decomposition of $I_A$ is 
\begin{equation*}
 I_A  = \langle x_1+x_2+2x_3,x_3^2 \rangle \cap \langle x_2, x_3\rangle. 
\end{equation*}
In other words, the eigenscheme of $A$ is the union of the
zero-dimensional subscheme of $\P^2$ with length $2$ defined by
$\langle x_1+x_2+2x_3,x_3^2\rangle$ and the point of $\P^2$ defined
by~$\langle x_2, x_3\rangle$.
The associated primes of $\QQ[x_1,x_2,x_3]/I_A$ are
$\langle x_1+x_2,x_3 \rangle$ and $\langle x_2,x_3 \rangle$. This
means that $V(I_A) = V(x_1+x_2,x_3) \cup V(x_2,x_3)$, which implies
that the eigenspaces corresponding to the eigenvalues $1$ and $2$ both
have dimension~$1$. Hence $A$ is not diagonalizable.
The projective line defined by $x_1+x_2+2x_3=0$ is the
scheme-theoretic linear span of the scheme defined by $ \langle
x_1+x_2+2x_3, x_3^2\rangle$.  This projective line coincides with the
generalized eigenspace of $A$ corresponding to the eigenvalue $1$.
Proposition~\ref{prop:hilbertFunction} shows that the degree of the
scheme of $\langle x_1+x_2+2x_3,x_3^2 \rangle$ corresponds to the size
of the Jordan block with eigenvalue $1$.
\end{example}
The example illustrates that the geometry of the eigenscheme encodes
information about the Jordan structure of a matrix.  One of the goals
of this paper is to find the decomposition of an eigenscheme into
irreducible subschemes and relate this to the Jordan canonical form of
the corresponding matrix. The following proposition shows that it is
sufficient to study eigenschemes of Jordan matrices.
\begin{proposition}
\label{prop:similar}
Let $A, B \in \K^{r \times r}$ be similar. Then the eigenschemes of $A$ and $B$ differ only by an automorphism of $\P^{r-1}$. 
\end{proposition}
\begin{proof}
Let $A \in \K^{r \times r}$, let $C \in GL(r, \K)$, and let $\bigwedge^2 C : \bigwedge^2 R_1^r \rightarrow \bigwedge^2 R_1^r$ be the linear transformation determined by sending $\boldsymbol f \wedge \boldsymbol g$ to $C\boldsymbol f \wedge C\boldsymbol g$. Then the ideal of $R$ generated by $\bigwedge^2 C(A\boldsymbol x \wedge \boldsymbol x) = CA\boldsymbol x \wedge C\boldsymbol x$ is equal to $I_A$. 

Now assume that $B \in \K^{r \times r}$ is similar to $A$, i.e., $B = C^{-1}AC$ for some $C \in GL(r, \K)$.  Then, as we saw above, the ideal generated by $I_B$ and the ideal generated by $\bigwedge^2 C(B\boldsymbol x \wedge \boldsymbol x) = CB\boldsymbol x \wedge C\boldsymbol x = AC \boldsymbol x \wedge C\boldsymbol x$ are the same. Hence the schemes of $I_A$ and $I_B$ differ only by an automorphism of $\P^{r-1}$.
\end{proof}
\section{Ideals of Jordan matrices with a single eigenvalue}
\label{sec:jordan-single}
Let $\lambda \in \K$ and $r \in \N$. Denote by $J_{\lambda,r}$ the Jordan block of size $r$ with eigenvalue~$\lambda$. For $k \in \N$, we write
\[
kJ_{\lambda,r} = \underbrace{J_{\lambda,r} \oplus \cdots \oplus J_{\lambda,r}}_k. 
\]
Let  $\ell \in \N$. For each $i \in \{1, \dots, \ell\}$, fix $k_i \in \N$ and $r_i \in \N$ with $r_1 > \cdots > r_\ell$.  A~Jordan matrix with a single eigenvalue can, up to a permutation of its blocks, be uniquely written as 
\[
J_\lambda = \bigoplus_{i=1}^\ell k_i J_{\lambda,r_i}, 
\]
We write $I_\lambda$ for $I_{J_\lambda}$.  In this section we describe
first a reduced Gr\"obner basis of $I_\lambda$ with respect to the
graded revlex order and then a primary decomposition of $I_\lambda$.
%
We also give a detailed description of each component ideal in the
primary decomposition of $I_\lambda$, including their Hilbert
polynomials. This description reveals the geometry of the eigenscheme
of~$J_\lambda$.
\subsection{Reduced Gr\"obner basis for $I_\lambda$}
Let 
\[
\Lambda = \bigcup_{i_1=1}^\ell \left\{\left. (i_1, i_2, i_3 ) \in \N^3 \ \right| \ 1 \leq i_2 \leq k_{i_1},  1 \leq i_3 \leq r_{i_1}\right\}.
\]
We totally order $\Lambda$ via $(i_1,i_2,i_3) > (j_1,j_2,j_3)$ if and
only if either $i_1 < j_1$, or $i_1=j_1$ and $i_2<j_2$, or $i_1=j_1$,
$i_2=j_2$ and $i_3<j_3$.  Let
$R = \K\left[x(\boldsymbol i) \ \left| \ \boldsymbol i \in
    \Lambda\right. \right]$
with $x(\boldsymbol i) > x(\boldsymbol j)$ if and only if
$\boldsymbol i > \boldsymbol j$. We denote by $\boldsymbol x$ the
vector of variables $x(\boldsymbol i)$, $\boldsymbol i \in \Lambda$.
%
  The indices of variables of $R$ correspond to the decomposition of a
  matrix into Jordan blocks as follows.  The first index $i_{1}$
  enumerates the different sizes of Jordan blocks that appear.  The
  second index $i_{2}$ enumerates the copies of Jordan blocks of
  size~$i_{1}$.  Finally, $i_{3}$ enumerates the rows in the Jordan
  block determined by $(i_{1},i_{2})$.
%
  Let
  $\Delta = \{(i_1,i_2) \in \N^2 \ | \ 1 \leq i_1
    \leq \ell, 1 \leq i_2 \leq k_{i_1}\}$
  be ordered as $\Lambda$ above. 
  Setting the degree of $x(i_1,i_2,i_3)$ to be the vector in
  $\N^{|\Delta|}$ whose only non-zero entry is $1$ in the
  $(i_1,i_2)$-th position, $R$ becomes a multi-graded ring.
  In this ring, there is a single multi-degree shared by the variables
  in one Jordan block, but these degrees are independent across
  blocks.

\begin{notation}
\label{not:f}
 If $\boldsymbol i = (i_1,i_2,i_3)  \in \Lambda$ and  $(i_1,i_2,i_3+1) \in \Lambda$, i.e., $i_3 \leq r_{i_1}-1$, then we write $\boldsymbol i^+$ for $(i_1,i_2,i_3+1)$.  Likewise, if $\boldsymbol i = (i_1,i_2,i_3)  \in \Lambda$ and if $(i_1,i_2,i_3-1) \in \Lambda$, i.e., $i_3 \geq 2$, then we write $\boldsymbol i^-$ for $(i_1,i_2,i_3-1)$. This means that $\boldsymbol i^{++}$ and $\boldsymbol i^{--}$ stand for $(i_1,i_2,i_3+2)$ and $(i_1,i_2,i_3-2)$ respectively. 
\end{notation}

Let $\varGamma = \{(\boldsymbol i, \boldsymbol j) \in \Lambda \times \Lambda \ | \ \boldsymbol i > \boldsymbol j\}$.  Consider the following subsets of $\varGamma$: 
\begin{eqnarray*}
 \varGamma_1 & = & \{((i_1,i_2,i_3),(j_1,j_2,j_3)) \in \varGamma  \ | \ 
1 \leq i_3 < r_{i_1}, 1\leq  j_3 < r_{j_1}, i_3 + j_3 \geq r_{j_1}+1\}, \\
 \varGamma_2 & = & \{((i_1,i_2,i_3),(j_1,j_2,j_3)) \in \varGamma  \ | \ 
 1 \leq i_3 < r_{i_1}, 1\leq  j_3 < r_{j_1}, i_3 + j_3 \leq r_{j_1}\}, \\
 \varGamma_3 & = & \{((i_1,i_2,i_3),(j_1,j_2,r_{j_1})) \in \varGamma \ | \ 1 \leq i_3 < r_{i_1}\},  \\
 \varGamma_4 & = & \{((i_1,i_2,r_{i_1}),(j_1,j_2,j_3)) \in \varGamma \ | \ 1 \leq j_3 < r_{j_1}\}. 
\end{eqnarray*}

Let $\boldsymbol i = (i_1,i_2,i_3) \in \Lambda$. Then 
\[
(J_\lambda \boldsymbol x)[\boldsymbol i, 1] = 
\left\{
\begin{array}{ll}
\lambda x(\boldsymbol i) + x(\boldsymbol i^+) & \mbox{if $i_3 < r_{i_1}$} \\
\lambda x (\boldsymbol i) & \mbox{if $i_3 = r_{i_1}$.}
\end{array}
\right. 
\]
Therefore, if $(\boldsymbol i, \boldsymbol j) \in \varGamma$, then the $(\boldsymbol i, \boldsymbol j)$-minor of $(J_\lambda \boldsymbol x \ | \  \boldsymbol x)$,  denoted $(J_\lambda \boldsymbol x \wedge \boldsymbol x)_{\boldsymbol i, \boldsymbol j}$, is: 
\[
(J_\lambda \boldsymbol x \wedge \boldsymbol x)_{\boldsymbol i, \boldsymbol j}
 = 
 \left|
 \begin{array}{cc}
 J_\lambda \boldsymbol x[\boldsymbol i, 1] & x(\boldsymbol i) \\
 J_\lambda \boldsymbol x[\boldsymbol j, 1] & x(\boldsymbol j)
 \end{array}
 \right|
 = 
 \left\{
 \begin{array}{ll}
 x(\boldsymbol i^+)x(\boldsymbol j) - x(\boldsymbol i)x(\boldsymbol j^+) & 
 \mbox{if $(\boldsymbol i,\boldsymbol j) \in \varGamma_1 \cup \varGamma_2$} \\
 x(\boldsymbol i^+)x(\boldsymbol j) & \mbox{if $(\boldsymbol i,\boldsymbol j) \in \varGamma_3$} \\
 -x(\boldsymbol i)x(\boldsymbol j^+) &\mbox{if $(\boldsymbol i,\boldsymbol j) \in \varGamma_4$} \\
 0 & \mbox{otherwise.}
 \end{array}
 \right. 
\]
\begin{proposition}
\label{prop:multi-graded} 
The ideal $I_\lambda$ of $J_\lambda$ is a multi-homogeneous ideal of $R$. 
\end{proposition}
\begin{proof}
It is enough to show that $(J_\lambda\boldsymbol x \wedge \boldsymbol
x)_{\boldsymbol i, \boldsymbol j}$ is multi-homogeneous for every
$(\boldsymbol i, \boldsymbol j) \in \bigcup_{i=1}^4
\varGamma_i$. Since every element of $\varGamma_3 \cup \varGamma_4$ is
a monomial, it is clearly multi-homogeneous. Every binomial
corresponding to an index pair in $\varGamma_1 \cup \varGamma_2$ can be written as 
\begin{equation}
\label{eq:binomial}
x(\boldsymbol i^+)x(\boldsymbol j)-x(\boldsymbol i)x(\boldsymbol j^+). 
\end{equation}
Each monomial appearing in (\ref{eq:binomial}) has the same multi-degree, and hence (\ref{eq:binomial}) also is multi-homogeneous. 
\end{proof}

For  simplicity, if $(\boldsymbol i, \boldsymbol j)  \in \varGamma_1 \cup \varGamma_2$, then we let
\begin{equation*}
f(\boldsymbol i, \boldsymbol j) \, := \,  (J_\lambda \boldsymbol x \wedge \boldsymbol x)_{\boldsymbol i, \boldsymbol j}
\, = \,  x(\boldsymbol i^+)x(\boldsymbol j) - x(\boldsymbol i)x(\boldsymbol j^+). 
\end{equation*}
Consider the following subsets of $R$:  
\begin{eqnarray*}
H_k &  = & \{f(\boldsymbol i, \boldsymbol j) \ | \ (\boldsymbol i, \boldsymbol j) \in \varGamma_k\}, \ \mbox{for $k \in \{1,2\}$}, \\  
H_3 & = & \{x(\boldsymbol i^+)x(\boldsymbol j)  \ |  \ (\boldsymbol i, \boldsymbol j) \in \varGamma_3\}, \\
H_4 & = & \{x(\boldsymbol i)x(\boldsymbol j^+) \ |  \ (\boldsymbol i, \boldsymbol j) \in \varGamma_4\}.  
\end{eqnarray*} 
Let $H = \bigcup_{k=1}^4 H_k$. Then it is immediate to see 
\[
I_\lambda = 
\left\{ 
\begin{array}{ll}
\langle 0 \rangle & \mbox{if  $r_1 = 1$} \\
\langle H \rangle & \mbox{otherwise.}
\end{array}
\right. 
\]
\begin{theorem}[Zaare-Nahandi and Zaare-Nahandi~\cite{zaare-nahandi}]
\label{thm:zaare-nahandi}
The set $H$ is a Gr\"obner basis for $I_\lambda$ with respect to the graded reverse lexicographic order~$>_{\mathrm {grevlex}}$. 
\end{theorem}
The following is an immediate consequence of Theorem~\ref{thm:zaare-nahandi}. 
\begin{corollary}
\label{cor:initialideal} 
The initial ideal $\mathrm{in}_{>_\mathrm{grevlex}} (I_\lambda)$ of $I_\lambda$ with respect to $>_\mathrm{grevlex}$ is 
\[
\left\langle x(\boldsymbol i^+)x(\boldsymbol j) \ | \ (\boldsymbol i, \boldsymbol j) \in \varGamma_1 \cup \varGamma_2 \cup \varGamma_3\right\rangle. 
\]
\end{corollary}
\begin{proof}
By Theorem~\ref{thm:zaare-nahandi}, $\mathrm{in}_{>_\mathrm{grevlex}} (I_\lambda)$ is generated by 
\[
\{x(\boldsymbol i^+)x(\boldsymbol j) \ | \ (\boldsymbol i, \boldsymbol j) \in \varGamma_1 \cup \varGamma_2 \cup \varGamma_3\}
\ \ 
\mbox{and} 
\ 
\left\{x(\boldsymbol i) x(\boldsymbol j^+) \ |\ (\boldsymbol i, \boldsymbol j) \in \varGamma_4\right\}. 
\]
Therefore, it is enough to show that if $(\boldsymbol i, \boldsymbol j) = ((i_1,i_2,r_{i_1}),(j_1,j_2,j_3))\in \varGamma_4$, then $x(\boldsymbol i)x(\boldsymbol j^+)$ is also the leading term of one of the polynomials in  $\bigcup_{i=1}^3 H_i$.  Since $r_{i_1}-1+j_3+1 >r_{i_1}$, we find $(\boldsymbol i^-,\boldsymbol j^+) \in \varGamma_1 \cup \varGamma_3$. This implies that  there exists a polynomial in $H_1 \cup H_3$ whose leading term coincides with $x(\boldsymbol i)x(\boldsymbol j^+)$, and hence we completed the proof. 
\end{proof}
\begin{remark}
\label{rem:minimal}
Corollary~\ref{cor:initialideal} implies that $\bigcup_{i=1}^3 H_i$ is a minimal Gr\"obner basis for $I_\lambda$ with respect to~$>_\mathrm{grevlex}$, but it is not a reduced Gr\"obner basis.  Indeed, if $(\boldsymbol i, \boldsymbol j) = ((i_1,i_2,i_3),(j_1,j_2,r_{j_1}-1)) \in \varGamma_1$ and if $i_3 >1$, then  
\[
x(\boldsymbol i^+) x(\boldsymbol j) - x(\boldsymbol i)x(\boldsymbol j^+) \in H_1. 
\]
Note that $(\boldsymbol i^-, \boldsymbol j^+) \in \varGamma_3$. Thus $x(\boldsymbol i)x(\boldsymbol j^+) \in H_3$. 

Also suppose that $i_3 \geq 2$ and  $(\boldsymbol i, \boldsymbol j) = ((i_1,i_2,i_3),(j_1,j_2,j_3)) \in \varGamma_2$. Then $(\boldsymbol i^-,\boldsymbol j^+) \in \varGamma_3$ if $j_3+1 = r_{j_1}$; $(\boldsymbol i^-,\boldsymbol j^+) \in \varGamma_2$ otherwise. 
This means that  the second term of 
\[
 x(\boldsymbol i^+)x(\boldsymbol j)-x(\boldsymbol i)x(\boldsymbol j^+) \in H_2 
\]
is divisible by $ x(\boldsymbol i)x(\boldsymbol j^+) \in H_3$ if $j_3+1 = r_{j_1}$; it is divisible by the leading term of 
\[
 x(\boldsymbol i)x(\boldsymbol j^+)-x(\boldsymbol i^-)x(\boldsymbol j^{++}) \in H_2
\]
otherwise. $\bigcup_{i=1}^3 H_i$ is therefore not reduced. In the rest of this subsection, we construct a reduced Gr\"obner basis for $I_\lambda$ with respect to $>_\mathrm{grevlex}$ from the minimal Gr\"obner basis $H$. 
\end{remark}
\begin{notation}
\label{not:g}
Let $(\boldsymbol i, \boldsymbol j) = ((i_1,i_2,i_3),(j_1,j_2,j_3)) \in \varGamma_2$. For simplicity, we denote $(i_1,i_2,1)$ by $\boldsymbol i^*$ and $(j_1,j_2,i_3+j_3)$ by $\boldsymbol j^*$. Then we let $g(\boldsymbol i, \boldsymbol j)$ be the binomial $x(\boldsymbol i^+)x(\boldsymbol j) - x(\boldsymbol i^*)x(\boldsymbol j^*)$. 
\end{notation}

Consider the following subsets of $R$: 
\[
 G_1  =  \left\{x(\boldsymbol i^+)x(\boldsymbol j) \ | \ \left(\boldsymbol i, \boldsymbol j\right) \in \varGamma_1 \cup \varGamma_3\right\} \ \  \mbox{and} \ \ G_2  =\left\{g(\boldsymbol i, \boldsymbol j)  \ | \  \left(\boldsymbol i, \boldsymbol j\right) \in \varGamma_2\right\}. 
\]
Define $G$ to be $\{0\}$ if $r_1 = 1$ and to be $G_1 \cup G_2$ otherwise. We prove that $G$ is a reduced Gr\"obner basis for $I_\lambda$ with respect to $>_\mathrm{grevlex}$.  The following two lemmas can be regarded as a reduction process that transforms $H$ to $G$. 
\begin{lemma}
\label{lem:monomial}
Let $(\boldsymbol i , \boldsymbol j)= ((i_1,i_2,i_3), (j_1,j_2,j_3)) \in \varGamma_1$. Then $x(\boldsymbol i^+)x(\boldsymbol j) \in I_\lambda$. 
\end{lemma}
\begin{proof}
As $j_3 < r_{j_1}$ and $i_3+j_3 \geq r_{j_1}+1$, we have $1\leq i_3 + j_3 - r_{j_1}   <  i_3  <  r_{i_1}$.
Thus $((i_1,i_2,i_3+j_3-r_{j_1}), (j_1,j_2,r_{j_1})) \in \varGamma_3$, and hence $x(i_1,i_2,i_3+j_3-r_{j_1}+1)x(j_1,j_2,r_{j_1}) \in H_3 \subseteq I_\lambda$. Since 
\[
x(i_1,i_2,i_3+1)x(\boldsymbol j) = 
\sum_{p=0}^{r_{j_1}-j_3-1}  f((i_1,i_2,i_3-p), (j_1,j_2,j_3+ p))+ x(i_1,i_2,i_3+j_3-r_{j_1}+1)x(j_1,j_2,r_{j_1})
\]
and since $ f((i_1,i_2,i_3-p), (j_1,j_2,j_3+ p)) \in I_\lambda$ for each $p \in \{0, \dots, r_{j_1}-j_3-1\}$, the monomial $x(i_1,i_2,i_3+1) x(j_1,j_2,j_3)$ is an element of  $I_\lambda$. 
\end{proof}
\begin{lemma}
\label{lem:generatingset}
The ideal $I_\lambda$ of $J_\lambda$ is generated by $G$. 
\end{lemma}
\begin{proof}
We can assume $r_1 > 1$, since otherwise $I_\lambda =  \langle 0 \rangle$ and $G = \{0\}$.
%
First, we show $I_\lambda \subseteq \langle G \rangle$. To do so, it is enough to prove that $H_k \subseteq \langle G \rangle$ for each $k \in \{1,2,3,4\}$. By definition, it is clear that $H_3 \subseteq G_1 \subseteq \langle G \rangle$. As was shown before,  $H_4 \subseteq G_1 \subseteq \langle G \rangle$. 

Suppose $(\boldsymbol i, \boldsymbol j) = \left((i_1,i_2,i_3),(j_1,j_2,j_3)\right) \in \varGamma_1$, which implies $x(\boldsymbol i^+)x(\boldsymbol j) \in G_1$.  Also, $i_3+j_3 \geq r_{j_1}+1$, and so if $i_3=1$, then $j_3 = r_{j_1}$, which is impossible. This implies $\left(\boldsymbol i^-, \boldsymbol j^+\right) \in \varGamma_1 \cup \varGamma_3$, because $i_3-1 \geq 1$ and $(i_3-1)+(j_3+1) = i_3+j_3 \geq r_{j_1}+1$. Hence $x(\boldsymbol i)x(\boldsymbol j^+) \in G_1$, and thus $f(\boldsymbol i, \boldsymbol j) \in \langle G_1 \rangle \subseteq \langle G \rangle$. Therefore, $H_1 \subseteq \langle G \rangle$, and it remains only to show that $H_2 \subseteq \langle G \rangle$. 

Let  $(\boldsymbol i, \boldsymbol j) = \left((i_1,i_2,i_3),(j_1,j_2,j_3)\right) \in \varGamma_2$. If $i_3 = 1$, then $f(\boldsymbol i, \boldsymbol j) = g(\boldsymbol i, \boldsymbol j) \in G_2$, and hence we may assume $i_3\geq 2$. As $i_3-1 \geq 1$ and $(i_3-1)+(j_3+1) = i_3+j_3  \leq r_{j_1}$, $\left(\boldsymbol i^-,\boldsymbol j^+\right) \in \varGamma_2$. Then $f(\boldsymbol i, \boldsymbol j) = g(\boldsymbol i, \boldsymbol j) - g\left(\boldsymbol i^-,\boldsymbol j^+\right)$. Thus $f(\boldsymbol i, \boldsymbol j) \in \langle G_2 \rangle \subseteq \langle G \rangle$. 

To prove $\langle G \rangle \subseteq I_\lambda$, it suffices to show that $G_1,G_2 \subseteq I_\lambda$. The containment $G_1 \subseteq I_\lambda$ follows immediately from the definition of $G_1$ and Lemma~\ref{lem:monomial}.  To show $G_2 \subseteq I_\lambda$, let $(\boldsymbol i, \boldsymbol j) \in \varGamma_2$. Then $\left((i_1,i_2,i_3-p),(j_1,j_2,j_3+p)\right) \in \varGamma_2$ for every $p \in \{0,1,\cdots, i_3-1\}$. It is easy to see that $g(\boldsymbol i, \boldsymbol j)$ is the sum of the binomials $f\left((i_1,i_2,i_3-p),(j_1,j_2,j_3+p)\right)$, $0 \leq p \leq i_3-1$. In particular, $g(\boldsymbol i, \boldsymbol j) \in \langle H_2 \rangle \subseteq I_\lambda$, and hence $G_2 \subseteq I_\lambda$. 
\end{proof}
\begin{proposition}
\label{prop:groebner}
The set $G$ is a reduced Gr\"obner basis for $I_\lambda$ with respect to $>_\mathrm{grevlex}$. 
\end{proposition}
\begin{proof}
From the construction of $G$ and Corollary~\ref{cor:initialideal}, it follows that the leading terms of the polynomials in $G$ generate $\mathrm{in}_{>_\mathrm{grevlex}} (I_\lambda)$. Thus $G$ is a Gr\"obner basis for $I_\lambda$ with the desired monomial order.  
The leading coefficient of each element of $G$ is $1$. Since $\varGamma_1 \cup \varGamma_3$ and $\varGamma_2$ are disjoint, no leading term of an element in $G$ divides any other such leading term. Furthermore, if $g$ is a binomial in $G$, i.e., $g(\boldsymbol i,\boldsymbol j) = x(\boldsymbol i^+)x(\boldsymbol j) - x(\boldsymbol i^*)x(\boldsymbol j^*)$ for some $(\boldsymbol i, \boldsymbol j) \in \varGamma_2$, then the non-leading term of $g$ is not divisible by the leading terms of the elements in $G$, because $((i_1,i_2,0),(j_1,j_2,i_3+j_3))\not\in\varGamma$. Therefore, $G$ is reduced. 
\end{proof}
The following corollary is a consequence of either Proposition~\ref{prop:groebner} or Remark~\ref{rem:minimal}.
\begin{corollary}
\label{cor:nondegenerate} 
The eigenscheme associated to $J_\lambda$ is non-degenerate; that is,
it is not a subscheme of any proper linear subspace of~$\P^{r-1}$.
\end{corollary}
\begin{proof}
If $r_1 = 1$, then $I_\lambda = \langle 0 \rangle$, and hence there is nothing to prove.
So assume that $r_1 \not= 1$. It is enough to show that the saturation of $I_\lambda$ with respect to the irrelevant ideal $\mathfrak{m} = \langle x(\boldsymbol i) \ | \ \boldsymbol i \in \Lambda \rangle$ of $R$ contains no linear forms.   

Suppose for the contradiction that $I_\lambda:\mathfrak{m}^\infty$ contains a linear form $L = \sum_{\boldsymbol i \in \Lambda} c(\boldsymbol i) x(\boldsymbol i)$ with $c(\boldsymbol i) \in \K$. Then there exists a positive integer $m$ such that $x(1,1,1)^m L \in I_\lambda$. Since $L$ is non-zero, there must be an $\boldsymbol i \in \Lambda$ such that $c(\boldsymbol i) \not=0$. Assume that $\boldsymbol i$ is the largest such element of $\Lambda$.  Then the leading monomial of $x(1,1,1)^mL$ is $x(1,1,1)^mx(\boldsymbol i)$. This monomial is not an element of $\mathrm{in}_{>_\mathrm{grevlex}} (I_\lambda)$, because it is not a multiple of the leading monomials of any element of $G$. This is a contradiction, and thus there are no linear forms contained in $I_\lambda:\mathfrak{m}^\infty$. 
\end{proof}
\begin{remark}
Let $A \in \K^{r \times r}$. Suppose that $\lambda \in \K$ is the only eigenvalue of $A$. Then Corollary~\ref{cor:nondegenerate} implies that the generalized eigenspace of $A$ corresponding to $\lambda$ coincides with the scheme-theoretic linear span of the eigenscheme of $A$.
\end{remark}
\subsection{Primary decomposition of $I_\lambda$}
The remainder of this section is devoted to a primary decomposition
of~$I_\lambda$ into $\ell$ ideals.
%
For $j \in \{1, \dots, \ell\}$ let 
\[
\Lambda_j = \{(i_1,i_2,i_3) \in \Lambda \ | \ 1 \leq i_1 \leq j\}.
\] 
Consider the following two subsets of $\Lambda$:  
\[
\Lambda_{j,1} = \{(i_1,i_2,i_3) \in \Lambda_j \ | \  \ r_{j}+1 \leq i_3 \leq r_{i_1}\}
\ \ 
\mbox{and} 
\ \ 
\Lambda_{j,2} =  \Lambda \setminus \Lambda_j. 
\]
Note that $\Lambda_{1,1}  = \Lambda_{\ell,2} = \emptyset$. 
We define two ideals of $R$ as follows: 
\[
I_{j,1} = 
\left\{
\begin{array}{ll}
\langle 0\rangle  & \mbox{if $j = 1$} \\
 \left\langle x(\boldsymbol i) \ \left| \ \boldsymbol i \in \Lambda_{j,1} \right. \right\rangle & \mbox{if $2 \leq j \leq \ell$} 
\end{array}
\right. 
 \ \ 
\mbox{and}   
\ \
I_{j,2} = 
\left\{
\begin{array}{ll}
 \left\langle x(\boldsymbol i) \ \left| \ \boldsymbol i \in \Lambda_{j,2} \right. \right\rangle & \mbox{if $1 \leq j \leq \ell-1$}  \\
\langle 0 \rangle 
& \mbox{if $j = \ell$.} 
\end{array}
\right. 
\]
Let $I_j$ be the sum of $I_{j,1}$ and $I_{j,2}$. The following example illustrates the notation.
\begin{example}
\label{ex:three-blocks}
Let  $\ell = 3$, $k_1 = k_2 = k_3 = 1$, $r_1=4$, $r_2 = 3$, and $r_3 = 2$, i.e., 
\[
J_\lambda = 
\left(
\begin{array}{cccc|ccc|cc}
\lambda & 1 & & & & & & &   \\
& \lambda & 1 & & & & & &  \\
& & \lambda & 1 & & & & & \\
& & & \lambda & & & & &  \\ \hline 
& & & & \lambda & 1 & & &  \\ 
& & & & & \lambda & 1 & &   \\
& & & & & & \lambda & &   \\ \hline 
& & & & & & & \lambda & 1 \\
& & & & & & & & \lambda 
\end{array}
\right). 
\]
Then 
\begin{eqnarray*}
 \Lambda_1 & = & \{(1,1,1), (1,1,2), (1,1,3),(1,1,4)\}, \\
 \Lambda_2 & = & \Lambda_1 \cup \{(2,1,1), (2,1,2),(2,1,3)\}, \\
 \Lambda_3 & = & \Lambda. 
\end{eqnarray*}
We thus obtain $\Lambda_{1,1}  = \Lambda_{3,2} = \emptyset$ and 
\begin{eqnarray*}
\Lambda_{1,2} & = & \{(2,1,1),(2,1,2),(2,1,3),(3,1,1),(3,1,2)\}, \\
\Lambda_{2,1} & = & \{(1,1,4)\}, \\
\Lambda_{2,2} & = & \{(3,1,1),(3,1,2)\}, \\
\Lambda_{3,1} & = & \{(1,1,3),(1,1,4),(2,1,3)\}.
\end{eqnarray*}
This implies $I_{1,1}  = I_{3,2} = \langle 0 \rangle$ and 
\begin{eqnarray*}
I_{1,2} & = & \left\langle x(2,1,1), x(2,1,2),x(2,1,3),x(3,1,1),x(3,1,2)\right\rangle, \\
I_{2,1} & = & \left\langle x(1,1,4)\right\rangle, \\
I_{2,2} & = &  \left\langle x(3,1,1), x(3,1,2)\right\rangle,\\
I_{3,1} & = &  \left\langle x(1,1,3), x(1,1,4),x(2,1,3) \right\rangle, 
\end{eqnarray*}
from which it follows that 
\begin{eqnarray*}
I_1 & = & \left\langle x(2,1,1), x(2,1,2),x(2,1,3),x(3,1,1),x(3,1,2)\right\rangle, \\
I_2  & = & \left\langle x(1,1,4), x(3,1,1), x(3,1,2)\right\rangle,\\
I_3 & = & \left\langle x(1,1,3),x(1,1,4),x(2,1,3)\right\rangle. 
\end{eqnarray*}
It is easy to check (e.g., with Macaulay2~\cite{M2}) that
$I_\lambda+I_1$, $I_\lambda+I_2$, and $I_\lambda+I_3$ are the three primary components in a primary
decomposition of $I_\lambda$.
\end{example}
Let $j \in \{1, \dots, \ell\}$ and let
$\mathfrak{q}_{\lambda,j} = I_\lambda+I_j$.
Example~\ref{ex:three-blocks} suggests that
$\bigcap_{j=1}^\ell \mathfrak{q}_{\lambda,j}$ is a primary
decomposition of $I_\lambda$.  The next proposition establishes that
the intersection is correct.
\begin{proposition}
\label{prop:decomposition}
$I_\lambda = \bigcap_{j=1}^\ell \mathfrak{q}_{\lambda,j}$. 
\end{proposition}
\begin{proof}
  We only need to show
  $I_\lambda \supseteq \bigcap_{j=1}^\ell (I_\lambda+I_j)$, as the
  other containment is obvious. It suffices to show that
\[
\bigcap_{j=1}^p \ (I_\lambda+I_j) = I_\lambda+I_{p,2}
\]
for every $p \in \{1, \dots, \ell\}$, because then $\bigcap_{j=1}^\ell (I_\lambda+I_j) = I_\lambda+I_{\ell,2} = I_\lambda + \langle 0\rangle = I_\lambda$. 
Since $I_{1,1} = \langle 0\rangle$, we have $I_1 = I_{1,2}$. Therefore, the statement is trivial for $p=1$. Suppose that $\bigcap_{j=1}^p (I_\lambda+I_j) = I_\lambda+I_{p,2}$ for some $p$. Then
\[
\bigcap_{j=1}^{p+1}\  (I_\lambda+I_j)  =  (I_\lambda+I_{p,2}) \cap (I_\lambda+I_{p+1}). 
\]
Since $( I_\lambda+I_{p,2}) \supseteq I_\lambda$ and $I_{p,2} \supseteq I_{p+1,2}$, the modular law for ideals implies 
\begin{eqnarray*}
( I_\lambda+I_{p,2}) \cap (I_\lambda+I_{p+1}) & = & I_\lambda + ( I_\lambda+I_{p,2}) \cap I_{p+1} \\
& = & I_\lambda +  ( I_\lambda+I_{p,2}) \cap (I_{p+1,1}+I_{p+1,2}) \\
& = & I_\lambda + I_{p+1,2} +  ( I_\lambda+I_{p,2}) \cap I_{p+1,1}. 
\end{eqnarray*}
Therefore, it is enough to show that 
\begin{equation}
\label{eq:subset}
(I_\lambda+I_{p,2}) \cap I_{p+1,1} \subseteq I_\lambda.
\end{equation}
To do so, we first prove that $(I_\lambda+I_{p,2}) \cap I_{p+1,1}  =  I_\lambda \cap I_{p+1,1} + I_{p,2} \cap I_{p+1,1}$. Clearly, $(I_\lambda+I_{p,2}) \cap I_{p+1,1}  \supseteq I_\lambda \cap I_{p+1,1} + I_{p,2} \cap I_{p+1,1}$. Thus we need to show the containment 
\begin{equation}
\label{eq:containment}
(I_\lambda+I_{p,2}) \cap I_{p+1,1}  \subseteq I_\lambda \cap I_{p+1,1} + I_{p,2} \cap I_{p+1,1}. 
\end{equation}

Denote by $\boldsymbol x_p$ the vector of variables $x(\boldsymbol i)$, $\boldsymbol i \in \Lambda_p$. Let $J_\lambda(p) = \bigoplus_{i=1}^{p} k_iJ_{\lambda,r_i}$ and $I_\lambda(p)$ the ideal generated by the entries of $J_\lambda(p) \boldsymbol x_p \wedge \boldsymbol x_p$. Then 
\[
I_\lambda+I_{p,2} = I_\lambda(p) + I_{p,2}. 
\]

Let $f_1 \in (I_\lambda(p) + I_{p,2}) \cap I_{p+1,1}$. Then there
exist $f_2 \in I_\lambda(p)$ and $h \in I_{p,2}$ such that
$f_1 = f_2+h$. Let $S = \K[x(\boldsymbol i) \ | \ \boldsymbol i \in \Lambda_p]$
and let $\{g_1, \dots, g_m\} \subset S[x(\boldsymbol i) \ | \ \boldsymbol i
\in \Lambda_{p,2}]$
be a generating set for~$I_\lambda(p)$.  
Let $f_2 = a_1g_1 +\cdots +a_mg_m$ with $a_1, \dots, a_m \in R$. Since
$R$ can be identified with
$S[x(\boldsymbol i) \ | \ \boldsymbol i \in \Lambda_{p,2}]$, by the
definition of $I_{p,2}$, we can write $a_i = a_i'+a_i''$ for each
$i \in \{1, \dots, m\}$, where $a_i' \in S$ and $a_i'' \in
I_{p,2}$. Therefore,
\[
f_2 = \sum_{i=1}^m a'g_i + \sum_{i=1}^m a_i''g_i  
\]
with $\sum_{i=1}^m a_i'g_i \in S$ and $\sum_{i=1}^m a_i''g_i \in I_\lambda(p) \cap I_{p,2}$. Thus we may assume $f_2 \in S \cap I_\lambda(p)$. 
 
Similarly, one can show that there exist $f_2' \in S \cap I_{p +1,1}$ and $h' \in I_{p,2} \cap I_{p+1,1}$ such that $f_1 = f_2'+h'$. Thus we get $f_1= f_2+h= f_2'+h'$, which implies $f_2-f_2' = h'-h \in I_{p,2}$. This means $f_2-f_2'=0$ (and hence $h -h' = 0$), because  $f_2-f_2' \in S$.  Thus, $f_2, f_2' \in I_\lambda(p) \cap I_{p+1,1}$ and $h, h' \in I_{p,2} \cap I_{p+1,1}$. Therefore, $f_1 = f_2+h \in I_\lambda(p)  \cap I_{p+1,1} + I_{p,2} \cap I_{p+1,1}$, from which 
 \[
 (I_\lambda+I_{p,2}) \cap I_{p+1,1} = (I_\lambda(p) + I_{p,2}) \cap I_{p+1,1} \subseteq  I_\lambda (p) \cap I_{p+1,1} + I_{p,2} \cap I_{p+1,1}
 \]
follows.  Because 
\[
I_\lambda (p) \cap I_{p+1,1} + I_{p,2} \cap I_{p+1,1} \subseteq I_\lambda  \cap I_{p+1,1} + I_{p,2} \cap I_{p+1,1},
\] 
we proved containment~(\ref{eq:containment}). Therefore, we get the equality 
\[
(I_\lambda+I_{p,2}) \cap I_{p+1,1} =  I_\lambda \cap I_{p+1,1} + I_{p,2} \cap I_{p+1,1}.
\] 

To prove (\ref{eq:subset}), it is sufficient to show that $I_{p,2} \cap I_{p+1,1} \subseteq I_\lambda$, because  $I_\lambda \cap I_{p+1,1}$ is clearly a subset of $I_\lambda$.  Since $\Lambda_{p+1,1}$ and $\Lambda_{p,2}$ are disjoint, we obtain 
\[
I_{p,2} \cap I_{p+1,1} = 
\left\langle  x(\boldsymbol i)x(\boldsymbol j) \ \left| \ \boldsymbol i \in \Lambda_{p+1,1}, \boldsymbol j \in \Lambda_{p,2}\right.\right\rangle.  
\]
Let $\boldsymbol i = (i_1,i_2,i_3) \in \Lambda_{p+1,1}$ and $\boldsymbol j = (j_1,j_2,j_3) \in \Lambda_{p,2}$. Then $i_3 \geq r_{p+1}+1$ and $j_1 >p$. Thus $r_{p+1} \geq r_{j_1}$, and hence  
\[
i_3+j_3 \geq r_{p+1} +1+1 >r_{j_1} +1, 
\]  
which implies that $i_3-1+j_3 \geq r_{j_1} +1$. Thus $(\boldsymbol i^-,\boldsymbol j) \in \varGamma_1 \cup \varGamma_3$, and hence the monomial $x(\boldsymbol i)x(\boldsymbol j)$ is contained in $I_\lambda$ by Lemma~\ref{lem:monomial}. Therefore, we proved $I_{p,2} \cap I_{p+1,1} \subseteq I_\lambda$, which completes the proof. 
\end{proof}
Our goal is Theorem~\ref{thm:primary} which shows that each ideal in
the decomposition in Proposition~\ref{prop:decomposition} is primary.
For this we need some preparation.
\begin{proposition}
\label{prop:reduced-groebner}
For each $j \in \{1, \dots, \ell\}$, let $G_j' = \{x(\boldsymbol i) \ | \ \boldsymbol i \in \Lambda_{j,1} \cup \Lambda_{j,2}\}$. Then $G \cup G_j'$ is a Gr\"obner basis for $\mathfrak{q}_{\lambda,j}$ with respect to $>_\mathrm{grevlex}$.  
\end{proposition}
\begin{proof}
  We us Notation~\ref{not:f} and Notation~\ref{not:g}.
  Since $G$ is a Gr\"obner basis for $I_\lambda$ with respect to
  $>_\mathrm{grevlex}$ and $G_1$ and $G'_j$ consist of monomials, it
  suffices to show that if $f_1 \in G_2$ and $f_2\in G_j'$, then
  $S(f_1,f_2)$ reduces to $0$ modulo $G \cup G_j'$.
  Let $f_1 \in G_2$ and let $f_2 \in G_j'$. Then there exist a
  $(\boldsymbol i, \boldsymbol j) = ((i_1,i_2,i_3),(j_1,j_2,j_3)) \in
  \varGamma_2$ such that
  \[
    f_1 = x(\boldsymbol i^+)x(\boldsymbol j)-x(\boldsymbol i^*)x(\boldsymbol j^*) 
  \]
  and $\boldsymbol a \in \Lambda_{j,1} \cup \Lambda_{j,2}$ such that
  $f_2 = x(\boldsymbol a)$.
  First assume that $\lcm(LT(f_1),LT(f_2)) = x(\boldsymbol i^+)x(\boldsymbol j)x(\boldsymbol a)$. Then 
  \[
    S(f_1,f_2) = -x(\boldsymbol i^*)x(\boldsymbol j^*)x(\boldsymbol a), 
  \] 
  and hence, it is a multiple of $x(\boldsymbol a)$.  Therefore
  $S(f_1,f_2)$ reduces to $0$ modulo $G \cup G_j'$.

  If $\boldsymbol j \in \Lambda_{j,2}$, then $S(f_1,f_2)$ is a
  multiple of an element of $G_j'$. We thus assume that
  $\boldsymbol i, \boldsymbol j \not\in \Lambda_{j,2}$.
  If $\boldsymbol a \in \Lambda_{j,2}$, then $LT(f_1)$ and $LT(f_2)$
  are relatively prime, or
  $\lcm(LT(f_1),LT(f_2)) = x(\boldsymbol i^+)x(\boldsymbol
  j)x(\boldsymbol a)$.
  Thus, we may assume that $\boldsymbol a \in \Lambda_{j,1}$.

  Next assume that
  $\lcm(LT(f_1),LT(f_2)) \not= x(\boldsymbol i^+)x(\boldsymbol
  j)x(\boldsymbol a)$.
  Then $\boldsymbol a \in \{\boldsymbol i^+,\boldsymbol j\}$, and thus
  $S(f_1,f_2) = -x(\boldsymbol i^*)x(\boldsymbol j^*)$.  If
  $\boldsymbol a = \boldsymbol i^+$, then
$
r_j+1=r_{i_1}-(r_{i_1}-r_j)+1 \leq i_3+1,  
$
while if $\boldsymbol a = \boldsymbol j$, then 
$r_j +1 \leq j_3$. In either case, $r_j +1 \leq i_3+j_3$, and hence
$\boldsymbol j^*= (j_1,j_2,i_3+j_3) \in \Lambda_{j,1}$. Therefore,
$S(f_1,f_2)$ reduces to $0$ modulo $G \cup G_j'$.
\end{proof}
To compute the Hilbert polynomial of $R/\mathfrak{q}_{\lambda,j}$, for
each $j \in \{1, \dots, \ell\}$, let:
\begin{equation*}
\Delta_j  = \left\{(i_1,i_2) \in \Delta \ | \ 1 \leq i_1 \leq j\right\}, \qquad 
\varTheta_j  = \left\{(i_1,i_2,1) \ \left| \ (i_1, i_2) \in \Delta_j\right.\right\}. 
\end{equation*}
\begin{proposition}
\label{prop:hilbertFunction}
Let $R_j = R/\mathfrak{q}_{\lambda,j}$. Then 
\[
H_{R_j}(t) = r_j\, {t+k_1+ \cdots k_j-1 \choose t}. 
\]
In particular, the scheme of $\mathfrak{q}_{\lambda,j}$ has dimension $k_1 +\cdots +k_j-1$ and degree~$r_j$. 
\end{proposition}
\begin{proof}
  For each monomial $\boldsymbol x$ of degree $t$ in $R$, there exists
  a unique non-increasing sequence
  $\boldsymbol i_1, \boldsymbol i_2, \cdots, \boldsymbol i_t \in
  \Lambda$
  such that
  $\boldsymbol x = x(\boldsymbol i_1)x(\boldsymbol i_2) \cdots
  x(\boldsymbol i_t)$.
  By Proposition~\ref{prop:reduced-groebner},
\begin{itemize}
\item[(*)] $\boldsymbol x \not\in \mathrm{in}_{>_\mathrm{grevlex}} (I_\lambda+I_j)$ if and only if $\boldsymbol i_1, \boldsymbol i_2, \cdots, \boldsymbol i_{t-1} \in \varTheta_j$ and $\boldsymbol i_t \in
\Lambda \setminus (\Lambda_{j,1} \cup \Lambda_{j,2})$. 
\end{itemize}
For each $(i_1,i_2) \in \Delta_j$, let 
\[
M_{(i_1,i_2)}(t) = \left\{\left. x(\boldsymbol i_1) \cdots x(\boldsymbol i_t) \not\in \mathrm{in}_{>_\mathrm{grevlex}} (I_\lambda+I_j)\ \right| \  \boldsymbol i_1 = (i_1,i_2,1)\right\}.
\]
and let $H_{(i_1,i_2)}(t) = \left|M_{(i_1,i_2)}(t)\right|$. Then 
\[
H_{R_j}(t) = \sum_{(i_1,i_2) \in \Delta_j} H_{(i_1,i_2)}(t). 
\]

For every $\boldsymbol x \in M_{(i_1,i_2)}(t)$, there exist a unique $(j_1,j_2) \in \Delta_j$ and a unique $\boldsymbol x' \in M_{(j_1,j_2)}(t-1)$ such that $\boldsymbol x = x(i_1,i_2,1)\boldsymbol x'$. This implies 
\[
H_{(i_1,i_2)}(t) = \sum_{(j_1,j_2) \in \Delta_j, \  (j_1,j_2) \leq (i_1,i_2)} H_{(j_1,j_2)}(t-1) 
\]
In particular, $H_{R_j}(t) = H_{(1,1)}(t+1)$. 
By induction, one can deduce that 
\[
H_{(i_1,i_2)}(t) = r_j\,  {t+k_{i_1}-i_2+k_{i_1+1}+ \cdots + k_j-1 \choose t-1}, 
\]
and hence obtain the desired equality
\[
H_{R_j}(t) = H_{(1,1)}(t+1) =  r_j \, {t+k_1+k_2+ \cdots + k_j-1 \choose t}. \qedhere
\]
\end{proof}
\begin{corollary}
\label{cor:regular} 
Let $j \in \{1, \dots, \ell\}$. Then, for every $\boldsymbol i \in \varTheta_j$, the element $x(\boldsymbol i) + \mathfrak{q}_{\lambda,j}$ of $R/\mathfrak{q}_{\lambda,j}$ is a non-zerodivisor. 
\end{corollary}
\begin{proof}
Let $h \in R$ be an arbitrary $t$-form and let $\boldsymbol i \in \varTheta_j$. We want to show that if 
\[
(h+\mathfrak{q}_{\lambda,j})(x(\boldsymbol i) + \mathfrak{q}_{\lambda,j}) = \mathfrak{q}_{\lambda,j}, 
\]
then $h \in \mathfrak{q}_{\lambda,j}$. By (*) in the proof of Proposition~\ref{prop:hilbertFunction}, the remainder of $h$ on division by $G \cup G_j'$ is a linear combination of monomials of the form $x(\boldsymbol i_1) \cdots x(\boldsymbol i_t)$, where $\boldsymbol i_1, \dots, \boldsymbol i_{t-1} \in \varTheta_j$ and $\boldsymbol i_t \in \Lambda \setminus (\Lambda_{j,1} \cup \Lambda_{j,2})$. Therefore, without loss of generality, we may assume that $h$ is a linear combination of such monomials. Note that $\boldsymbol i_1 \geq \cdots \geq  \boldsymbol i_{t}$.  Thus there is no polynomial in $\mathfrak{q}_{\lambda,j}$ whose leading monomial divides that of $h \cdot x(\boldsymbol i)$, from which it follows that $h=0$. Therefore, we completed the proof. 
\end{proof}
For each $j \in \{1, \dots, \ell\}$, let $\mathfrak{h}_j = \langle x(\boldsymbol i) \ | \ \boldsymbol i \in \Lambda \setminus \varTheta_j \rangle$. Proposition~\ref{prop:radical} and Corollary~\ref{cor:radical} compute the radical of $\mathfrak{q}_{\lambda,j}$. 

\begin{proposition}
\label{prop:radical}
The radical of $I_\lambda$ is $\mathfrak{h}_\ell$. 
\end{proposition}
\begin{proof} 
  Note that $\mathfrak{h}_\ell$ is the ideal of the eigenspace of
  $J_\lambda$.  If $\K$ is algebraically closed, since the zero-set of
  $I_\lambda$ is the eigenspace, it follows from the Nullstellensatz
  that $\sqrt{I_\lambda} =\mathfrak{h}_\ell$.  If $\K$ is arbitrary
  then a combinatorial proof can be given along the lines of the proof
  of the previous statements in this section.  With
  $\boldsymbol i = (i_1,i_2,i_3) \in \Delta \setminus \varTheta_\ell$,
  one can show that if $i_3 \geq \frac{r_{i_1}+2}{2}$, then
  $x(\boldsymbol i)^2 \in I_\lambda$; while if
  $i_3 \leq\frac{r_{i_1}+1}{2}$, then
  $x(\boldsymbol i)^\beta \in I_\lambda$ with
  $\beta = \left\lfloor\frac{r_{i_1}-2i_3+1}{i_3-1}\right\rfloor$. We
  omit, however, the detailed proof for the sake of brevity.
\end{proof}
\begin{corollary}
\label{cor:radical}
For each $i \in \{1, \dots, \ell\}$, the radical of $\mathfrak{q}_{\lambda,j}$ is $\mathfrak{h}_j$. 
\end{corollary}
\begin{proof}
It is well-known that
\[
\sqrt{I_\lambda+I_j} = \sqrt{\sqrt{I_\lambda}+\sqrt{I_j}}.
\]
Recall that  $\sqrt{I_\lambda} =  \langle x(\boldsymbol i) \, | \, \boldsymbol i \in \Lambda \setminus \varTheta_\ell\rangle$ and $\sqrt{I_j} = I_j =  \langle x(\boldsymbol i) \, | \, \boldsymbol i \in \Lambda_{j,1} \cup \Lambda_{j,2} \rangle$. Since $\Lambda_{j,1} \subseteq \Lambda \setminus \varTheta_\ell$ and $\varTheta_\ell \setminus \Lambda_{j,2} = \varTheta_j$, we get
\begin{eqnarray*}
\sqrt{I_\lambda}+\sqrt{I_j} & = & \left\langle x(\boldsymbol i) \ \left| \ \boldsymbol i \in (\Lambda \setminus \varTheta_\ell) \cup \Lambda_{j,1} \cup \Lambda_{j,2}\right. \right\rangle \\
& = &  \left\langle x(\boldsymbol i) \ \left| \ \boldsymbol i \in \Lambda \setminus \varTheta_j \right. \right\rangle,   
\end{eqnarray*}
which completes the proof. 
\end{proof}
The following corollary shows that, for each
$j\in \{1, \dots, \ell\}$, $\mathfrak{q}_{\lambda,j}$ is
\emph{cellular}, i.e., each variable of $R$ is either a
non-zerodivisor or nilpotent in $R/\mathfrak{q}_{\lambda,j}$ (see
\cite[Section~6]{eisenbud-sturmfels}).
\begin{corollary}
For each $j \in \{1, \dots, \ell\}$, the ideal $\mathfrak{q}_{\lambda,j}$ is cellular. 
\end{corollary}
\begin{proof}
  A monomial is a non-zerodivisor modulo $\mathfrak{q}_{\lambda,j}$ if
  and only if all its variables are such, and it is nilpotent if and only
  if one of its variables is nilpotent.  Corollary~\ref{cor:regular} shows that
  $x(\boldsymbol i) + \mathfrak{q}_{\lambda,j}$ is a non-zerodivisor
  for every for every $\boldsymbol i$ of $\varTheta_j$.  
  Therefore, it is enough to prove 
  that $x(\boldsymbol i)$ is nilpotent for every
  $\boldsymbol i \not\in \varTheta_j$.
  The case $\boldsymbol i \not\in \varTheta_\ell$ is
  Proposition~\ref{prop:radical}, and if
  $\boldsymbol i \in \Lambda \setminus \varTheta_j$, but
  $\boldsymbol i \not\in \Lambda \setminus \varTheta_\ell$, then
  $x(\boldsymbol i) \in I_j \subseteq \mathfrak{q}_{\lambda,j}$. In
  particular, $x(\boldsymbol i) + \mathfrak{q}_{\lambda,j}$ is
  nilpotent in $R/\mathfrak{q}_{\lambda,j}$. 
\end{proof}
\begin{theorem}
\label{thm:primary}
For each $j \in \{1, \dots, \ell\}$, $\mathfrak{q}_{\lambda,j}$ is primary. In particular, the decomposition of $I_\lambda$ in Proposition~\ref{prop:decomposition} is a primary decomposition. 
\end{theorem}
\begin{proof}
  Theorem~8.1 in \cite{eisenbud-sturmfels} shows that the binomial
  part of every associated prime of $\mathfrak{q}_{\lambda,j}$ is a
  prime lattice ideal associated to the elimination ideal
\begin{equation}
\label{eq:elimination}
\left(\mathfrak{q}_{\lambda,j} : x(\boldsymbol a)  \right) \cap \K\left[x(\boldsymbol i) \ \left| \ \boldsymbol i \in \varTheta_j\right. \right],
\end{equation}
for some
$\boldsymbol a \in \Lambda \setminus (\varTheta_j \cup \Lambda_{j,1}
\cup \Lambda_{j,2})$.
Since $I_\lambda$ and $I_j$ are multi-homogeneous, so is
$\mathfrak{q}_{\lambda,j}$ and hence
$\mathfrak{q}_{\lambda,j} : x(\boldsymbol a)$.  This implies that the
ideal~(\ref{eq:elimination}) is multi-homogeneous in
$\K[x(\boldsymbol i) \ | \ \boldsymbol i \in \varTheta_j$], with its
multi-grading inherited from~$R$.
Now, every multi-homogeneous ideal in
$\K[x(\boldsymbol i) \ | \ \boldsymbol i \in \varTheta_j]$ is
monomial, because there is only one variable in each multi-degree, and
the multi-degrees are linearly independent.  Consequently the binomial
parts of all associated primes of $\mathfrak{q}_{\lambda,j}$ are zero
and thus $\mathfrak{h}_{j}$ is its only associated prime.
\end{proof}
\begin{remark}
  In general binomial primary decomposition is sensitive to the
  characteristic of~$\K$.  For example, $\langle x^{p}-1\rangle$ is
  primary in characteristic $p$, but factors in all other
  characteristics.  The primary decomposition in
  Theorem~\ref{thm:primary} is valid in every charactersistic and
  without an algebraically closedness assumption because all appearing
  characters of associated primes are zero. In this case
  \cite[Corollary~2.2]{eisenbud-sturmfels} applies even without the
  algebraically closedness assumption.
\end{remark}
\begin{remark}
  The primary decomposition in Theorem~4.22 is also a mesoprimary
  decomposition according to \cite[Definition~13.1]{KM2014}.  All
  occurring ideals are mesoprimary since they are primary over~$\CC$
  by \cite[Corollary~10.7]{KM2014}.  The decomposition itself is
  mesoprimary since all occurring associated mesoprimes are equal to
  one of the $\mathfrak{h}_j$ by the multi-homogeneity argument in the
  proof of Theorem~\ref{thm:primary}.  However, the decomposition is
  not a combinatorial mesoprimary decomposition, essentially because
  the intersection of the monomial parts of the components are not
  aligned.
\end{remark}
The following example illustrates how the information in the primary
decomposition is sufficient to reconstruct the Jordan structure of a
matrix with a single eigenvalue.
\begin{example}
  Let $A \in \K^{17 \times 17}$ with a single eigenvalue
  $\lambda \in \K$, and
%
  suppose that the primary decomposition of
  $I _A= \bigcap_{i=1}^3 I_i$ has three components.
  Suppose further that the Hilbert functions of the components satisfy
\begin{equation}
\label{eq:hilbertFunction}
H_{R/I_i}(t) = 
\left\{
\begin{array}{ll}
4(t+1) & \mbox{$i=1$} \\
3 {t+2 \choose 2} & \mbox{$i=2$} \\
2 {t+5 \choose 5} & \mbox{$i=3$.}
\end{array} 
\right. 
\end{equation}
Then, by Proposition~\ref{prop:hilbertFunction}, it follows that the
Jordan block decomposition
$J_\lambda = \bigoplus_{i=1}^\ell k_i J_{\lambda,r_i}$ has $\ell = 3$,
$r_1 = 4$, $r_2 = 3$, and $r_3= 2$.
Additionally (\ref{eq:hilbertFunction}) yields the following linear
relations among $k_1$, $k_2$, and $k_3$:
\[
\left\{ 
\begin{array}{ccccccl}
k_1 & & & & & = & 2 \\
k_1& + & k_2 & & & = & 3 \\
k_1& + & k_2 & + & k_3 & = & 6,  
\end{array}
\right.
\]
from which it follows that $k_1 = 2$, $k_2 = 1$, and $k_3 = 3$. This means that $J_\lambda$ consists of two Jordan blocks of size $4$, one Jordan block of size $3$, and three Jordan blocks of size $2$. 
\end{example}
The following corollary explains when a square matrix is
diagonalizable from a new commutative algebra point of view.
\begin{corollary}
\label{cor:diagonalizability}
Let $A \in \K^{r \times r}$ have a single eigenvalue, which lies
in~$\K$. Then $A$ is diagonalizable if and only if $I_A$ is radical.
\end{corollary}
\begin{proof}
Let $A$ have eigenvalue $\lambda$ and Jordan canonical form~$J_\lambda  = \bigoplus_{i=1}^\ell k_i J_{\lambda,r_i}$.
Then $A$ is diagonalizable if and only if the Jordan blocks in $J_\lambda$ are all of size $1$. So, by Proposition~\ref{prop:similar}, it is sufficient to show that the latter condition is equivalent to radicality of~$I_\lambda$. 

If the Jordan blocks in $J_\lambda$ are all of size $1$, then $I_\lambda$ is the zero ideal and hence it is radical.

Now assume that $I_\lambda$ is radical. Proposition~\ref{prop:radical} implies that  $I_{\lambda} = \mathfrak{h}_\ell$. By the definition of $\mathfrak{h}_\ell$, and the fact that $I_{\lambda}$ does not contain any linear polynomials, we must have that $\varTheta_\ell=\Lambda$. Hence all the Jordan blocks have size $1$.
\end{proof}
We close this section by stating the following geometric property of the zero-dimensional eigenscheme of a square matrix with a single eigenvalue. 
\begin{proposition}
\label{prop:curvilinear}
Let $r \geq 2$ and let $A \in \mathbb{K}^{r \times r}$ with a single
eigenvalue, which lies in~$\K$. If the eigenscheme of $A$ is
zero-dimensional, then it is curvilinear.
\end{proposition}
\begin{proof}
  Let $Z_A$ be the eigenscheme of $A$.  It suffices to show that $Z_A$
  can be embedded in a non-singular curve.  This is trivial if
  $r = 2$, so suppose $r \geq 3$.
  By Proposition~\ref{prop:similar}, we may assume that $A$ is in
  Jordan canonical form. Since the eigenspace of $A$ has dimension
  $1$, there is only a single Jordan block. An easy calculation shows
  that the ideal of $Z_A$ is generated by the $2 \times 2$ minors of
  the $2 \times r$ matrix $\left(
  \begin{smallmatrix}
    x_{0} & x_1 & \dots & x_{r-2} & x_{r-1} \\
    x_{1} & x_{2} & \dots & x_{r-1} & 0
  \end{smallmatrix}
\right)$.
The $2 \times 2$ minors of the submatrix consisting of the first $r-1$
columns generate the ideal of a rational normal curve
in~$\mathbb{P}^{r-1}$~\cite[Example~1.16]{harris}.  Thus this rational
normal curve contains $Z_A$ as a subscheme.
\end{proof}
\section{Ideals of general Jordan matrices}
\label{sec:jordan}
Let $J$ be a Jordan matrix with at least two distinct eigenvalues.  In
this section we find a primary decomposition of the ideal of~$J$. We
begin with the case that $J$ has exactly two distinct eigenvalues.
The general case is obtained by induction.

Let $A \in \K^{r \times r}$ and let $B \in \K^{s \times s}$. Let $R = \K[x_1, \dots, x_r, y_1, \dots, y_s]$ be the bi-graded polynomial ring with bi-grading given as  follows:  
\[
\left\{
\begin{array}{ll}
\deg(x_i) = (1,0) & \mbox{for $i \in \{1, \dots, r\},$}  \\
\deg(y_j) = (0,1) & \mbox{for $j \in \{1, \dots, s\}$}.  
\end{array}
\right.
\]
The bi-graded piece of $R$ in bi-degree $(1,1)$ is denoted by $R_{(1,1)}$. Let $\boldsymbol{x}$ and $\boldsymbol{y}$ be the vectors of variables $x_1, \dots, x_r$ and $y_1, \dots,y_s$ respectively.  We denote the sets of entries $A\boldsymbol{x} \wedge \boldsymbol{x}$, $B\boldsymbol{y} \wedge \boldsymbol{y}$, and 
\[
(A \oplus B) \left(\begin{array}{c} \boldsymbol x \\ \boldsymbol y\end{array}\right) \wedge \left(\begin{array}{c} \boldsymbol x \\ \boldsymbol y\end{array}\right) 
\]
by $L_A$, $L_B$, and $L_{A \oplus B}$ respectively. Let $I_{A \oplus B} = \langle L_{A \oplus B} \rangle$. Recall that $A[\alpha,:]$ denotes the $\alpha$-th row of $A$. By definition of $L_{A \oplus B}$, we get
\begin{equation}
\label{eq:idealEquality}
I_{A \oplus B} = \langle L_A \rangle + \langle L_B \rangle +  \left\langle \left. 
 A[\alpha, :]\boldsymbol x \, y_\beta-B[\beta, :]\boldsymbol y \, x_\alpha
  \ 
  \right| 
  \
  (\alpha, \beta) \in \varOmega
\right\rangle, 
\end{equation}
where $\varOmega = \{(i,j) \ | \ 1 \leq i \leq r, 1 \leq j \leq
s\}$. Put $\varOmega$ in the ordering defined by
$(i_1,j_1)>(i_2,j_2)$ when $i_1<i_2$ or $i_1=i_2$ and $j_1<j_2$. Let $T = \{x_iy_j \ | \ (i,j) \in \varOmega\}$, and note that this is a basis for $R_{(1,1)}$. Let $\varPhi$ be the $|\varOmega| \times
|\varOmega|$ matrix whose $(\alpha, \beta)$-th row is the
coordinate vector of
\[
A[\alpha, :]\boldsymbol x \, y_\beta-B[\beta, :]\boldsymbol y \, x_\alpha
\] 
with respect to $T$, i.e., the $((\alpha,\beta),(i,j))$-entry of $\varPhi$ is 
\begin{equation}
\label{eq:Phi}
\varPhi[(\alpha,\beta), (i,j)] = 
\left\{ 
\begin{array}{ll}
A[\alpha,\alpha]-B[\beta,\beta] & \mbox{if $(\alpha,\beta) = (i,j)$,} \\
-B[\beta,j] & \mbox{if $i = \alpha$ and $j \not= \beta$,} \\
A[\alpha, i] & \mbox{if $i \not=\alpha$ and $ j=\beta$,} \\
0 & \mbox{otherwise.}
\end{array}
\right. 
\end{equation}
\begin{example}
\label{ex:direct-sum}
Consider the following two matrices: 
\[
A = 
\left(
\begin{array}{cc}
-1 & 4 \\
-1 & 3 
\end{array}
\right) \ \mbox{and} \ 
B = 
\left(
\begin{array}{cc}
-7 & 9 \\
-4 & 5 
\end{array}
\right). 
\] 
Let $R = \K[x_1,x_2,y_1,y_2]$, let $\boldsymbol x = \left(\begin{array}{rr} x_1 & x_2\end{array}\right)^T$, and let $\boldsymbol y = \left(\begin{array}{rr} y_1 & y_2\end{array}\right)^T$.  Then 
\[
\left| \begin{array}{ll} 
  A[\alpha, :] \boldsymbol{x} & x_\alpha \\
  B[\beta, :] \boldsymbol{y} & y_\beta
  \end{array}
  \right|
  = \left\{
  \begin{array}{ll}
  6x_1y_1-9x_1y_2 +4x_2y_1 & \mbox{if $(\alpha,\beta) = (1,1)$,} \\
  4x_1y_1-6x_1y_2+4x_2y_2 & \mbox{if $(\alpha,\beta) = (1,2)$,} \\
  -x_1y_1+10x_2y_1-9x_2y_2 & \mbox{if $(\alpha,\beta) = (2,1)$,} \\
  -x_1y_2+4x_2y_1-2x_2y_2 & \mbox{if $(\alpha,\beta) = (2,2)$.} 
  \end{array}
  \right.
\]
Therefore, we obtain 
\[
\varPhi = 
\left(
\begin{array}{rrrr}
6 & -9 & 4 & 0 \\
4 & -6 & 0 & 4 \\
-1 & 0 & 10 & -9 \\
0 & -1 & 4 & -2 
\end{array}
\right). 
\]
Since $\det \varPhi = 16$, the rank of $\varPhi$ is $4$. 
\end{example}
\begin{proposition}
\label{prop:induction}
If $\rank (\varPhi) = rs$, then  
$
I_{A \oplus B} = \langle L_A, y_1, \dots, y_s \rangle \cap \langle L_B, x_1, \dots, x_r\rangle.
$ 
\end{proposition}
\begin{proof}
%
First we show the following set equality: 
\begin{equation}
\label{eq:setEquality1}
 \langle L_A, y_1, \dots, y_s \rangle \cap \langle L_B, x_1, \dots, x_r\rangle =  \langle L_A \rangle + \langle L_B \rangle + \left\langle \left. x_iy_j \ \right| \ (i,j) \in \varOmega \right\rangle. 
\end{equation}
To prove (\ref{eq:setEquality1}), it is equivalent to show that 
\begin{equation}
\label{eq:setEquality2}
\langle L_A, y_1, \dots, y_s \rangle \cap \langle L_B, x_1, \dots, x_r\rangle  
=  \langle L_A \rangle  + \langle L_B \rangle + \langle x_1, \dots, x_r \rangle \cap \langle y_1, \dots, y_s \rangle. 
\end{equation}

The containment ``$\supseteq$'' follows immediately from the modular law for ideals plus the fact that $\langle L_A \rangle \subseteq \langle x_1, \dots, x_r \rangle$ and $\langle L_B \rangle \subseteq \langle y_1, \dots, y_s \rangle$.   

For the other containment let $f  \in  \langle L_A, y_1, \dots, y_s \rangle \cap \langle L_B, x_1, \dots, x_r\rangle$ be arbitrary. Then, as we saw in the proof of Proposition~\ref{prop:decomposition},  there exist a $g \in \langle L_A \rangle$, an $h \in \langle y_1, \dots, y_s \rangle$, a $g' \in \langle L_B \rangle$, and an $h' \in  \langle x_1, \dots, x_r\rangle$ such that $f = g+h = g'+h'$. Since $\langle L_A \rangle \subseteq \langle x_1, \dots, x_r \rangle$ and $\langle L_B \rangle \subseteq \langle y_1, \dots, y_s \rangle$, we have 
\[
g-h' = g'-h \in   \langle x_1, \dots, x_r \rangle \cap  \langle y_1, \dots, y_s \rangle, 
\]
which means that there exists a $p \in  \langle x_1, \dots, x_r \rangle \cap  \langle y_1, \dots, y_s \rangle$ such that $g'-h = p$. Thus 
\[
h = g' - p \in \langle L_B \rangle +   \langle x_1, \dots, x_r \rangle \cap  \langle y_1, \dots, y_s \rangle.
\] 
Therefore, 
\[
f  = g+h \in \langle L_A \rangle + \langle L_B \rangle +   \langle x_1, \dots, x_r \rangle \cap  \langle y_1, \dots, y_s \rangle, 
\] 
which proves set equality~(\ref{eq:setEquality2}). 

Recall that $T$ is a basis for $R_{(1,1)}$. 
By assumption, 
\[ 
 \left\{ \left. 
A[\alpha, :]\boldsymbol x \, y_\beta-B[\beta, :]\boldsymbol y \, x_\alpha
  \ 
  \right| 
  \
  (\alpha, \beta) \in \varOmega
\right\}
\]
is also a basis for $R_{(1,1)}$, and hence 
\[
\left\langle \left. x_iy_j \ \right| \ (i,j) \in \varOmega \right\rangle \ \mbox{and} \ 
\left\langle \left. 
A[\alpha, :]\boldsymbol x \, y_\beta-B[\beta, :]\boldsymbol y \, x_\alpha
  \ 
  \right| 
  \
  (\alpha, \beta) \in \varOmega
\right\rangle
\] 
are the same ideal. Therefore, the proposition follows from~(\ref{eq:idealEquality}) and~(\ref{eq:setEquality1}).  
\end{proof}
\begin{example}
Let $A$, $B$, $R$, $\boldsymbol x$, $\boldsymbol y$, and $\varPhi$ be as given in Example~\ref{ex:direct-sum}. Since $\rank \varPhi = 2 \cdot 2 = 4$, it follows from Proposition~\ref{prop:induction} that $I_{A \oplus B}$ can be written as follows: 
\begin{equation}
\label{eq:primary-decomposition-example}
\left\langle 
(-x_1+2x_2)^2 , y_1, y_2 
\right\rangle 
\cap 
\left\langle 
(2y_1-3y_2)^2, x_1, x_2
\right\rangle,  
\end{equation}
where 
\begin{eqnarray*}
\det (A \boldsymbol x \ | \ \boldsymbol x) &  = & (-x_1+2x_2)^2, \\
\det (B \boldsymbol x \ | \ \boldsymbol x)  & = & (2y_1-3y_2)^2.
\end{eqnarray*} 

It is not hard to show that (\ref{eq:primary-decomposition-example}) is a primary decomposition of $I_{A \oplus B}$. This means that the eigenscheme of $A \oplus B$ is the union of two primary components, both of which are of dimension $0$ and degree $2$. The support of this eigenscheme is the union of two points of $\P^3$ defined by $-x_1+2x_2=y_1=y_2=0$ and $2y_1-3y_2=x_1=x_2=0$. 
\end{example}
Let $\lambda, \mu \in \K$ be distinct and let $J_\lambda$ and $J_\mu$ be Jordan matrices with eigenvalues $\lambda$ and $\mu$ respectively. For simplicity, $I_{\lambda, \mu}$ denotes $I_{J_\lambda \oplus J_\mu}$. Since the Jordan canonical form of a matrix is unique up to the order of the Jordan blocks, we may assume that there exist $m,n \in \N$, $(k_1, \dots, k_m) \in \N^m$, $(k_1', \dots, k_n') \in \N^n$, and strictly decreasing sequences of positive integers $(r_1,\dots,  r_m)$ and $(r_1', \dots, r_n')$ such that 
\[
J_\lambda = \bigoplus_{i=1}^m k_i J_{\lambda,r_i} \ 
\mbox{and} \ \ 
J_\mu = \bigoplus_{i=1}^n k_i' J_{\mu,r_i'}. 
\] 

Let $\xi_\lambda = \sum_{i=1}^m k_ir_i$, let $\xi_\mu =  \sum_{i=1}^n k_i'r_i'$, and let $R = \K[x_1, \dots,x_{\xi_\lambda}, y_1, \dots, y_{\xi_\mu}]$.  Consider the vector $\boldsymbol x$ of variables $x_1, \dots, x_{\xi_\lambda}$ and the vector $\boldsymbol y$ of variables $y_1, \dots, y_{\xi_\mu}$. Write $L_\lambda$ for the set of entries of $J_\lambda \boldsymbol x \wedge \boldsymbol x$ and $L_\mu$ for those of~$J_\mu \boldsymbol y \wedge \boldsymbol y$. 
\begin{proposition}
\label{prop:initial}
Let $\lambda, \mu \in \K$ be distinct. Then
\[
I_{J_\lambda \oplus J_\mu} = \langle L_\lambda, y_1, \dots, y_{\xi_\mu}\rangle \cap 
\langle L_\mu, x_1, \dots, x_{\xi_\lambda}\rangle.
\] 
\end{proposition}
\begin{proof}
Define 
\[
 \Pi_\lambda= \bigcup_{i=1}^m \left\{cr_i \ | \ 1 \leq c \leq k_i\right\}
\ 
\mbox{and} \  \
 \Pi_\mu = \bigcup_{i=1}^n \left\{cr_i' \ | \ 1 \leq c \leq k_i'\right\}. 
\]
Then 
\begin{equation}
\label{eq:entries}
\left|
\begin{array}{ll}
J_\lambda \boldsymbol x[\alpha,:] & x_\alpha \\
J_\mu \boldsymbol y[\beta,:] & y_\beta
\end{array}
\right|
= 
\left\{
\begin{array}{ll}
(\lambda-\mu) x_\alpha y_\beta - x_\alpha y_{\beta+1} + x_{\alpha+1}y_\beta & \mbox{if $\alpha \not\in \Pi_\lambda$ and $\beta \not\in \Pi_\mu$}, \\
(\lambda-\mu) x_\alpha y_\beta - x_\alpha y_{\beta+1} & \mbox{if $\alpha \in \Pi_\lambda$ and $\beta \not\in \Pi_\mu$}, \\
(\lambda-\mu) x_\alpha y_\beta+ x_{\alpha+1}y_\beta & \mbox{if $\alpha \not\in \Pi_\lambda$ and $\beta \in \Pi_\mu$}, \\
0 & \mbox{otherwise.} 
\end{array}
\right. 
\end{equation}
Therefore, the $((\alpha,\beta),(i,j))$-entry of the $\xi_\lambda \xi_\mu \times \xi_\lambda \xi_\mu$ matrix~(\ref{eq:Phi}) for $A = J_\lambda$ and $B = J_\mu$ is 
\begin{equation}
\label{eq:Phigeneral}
\varPhi[(\alpha,\beta),(i,j)] = 
\left\{
\begin{array}{ll}
\lambda-\mu & \mbox{if $(\alpha,\beta)=(i,j)$,} \\
-1 & \mbox{if $\alpha  \in \Pi_\lambda$, $\beta \not\in \Pi_\mu$ and $(\alpha, \beta+1) = (i,j)$,} \\
-1 & \mbox{if $\alpha  \not\in \Pi_\lambda$, $\beta\not\in \Pi_\mu$ and $(\alpha, \beta+1)= (i,j)$,} \\
1 & \mbox{if $\alpha  \not\in \Pi_\lambda$, $\beta \in \Pi_\mu$ and $(\alpha+1, \beta)= (i,j)$,} \\
1 &  \mbox{if $\alpha  \not\in \Pi_\lambda$, $\beta\not\in \Pi_\mu$ and $(\alpha+1, \beta)=(i,j)$,} \\
0 & \mbox{otherwise} 
\end{array}
\right. 
\end{equation}
by~(\ref{eq:entries}). Since $(\alpha,\beta) > (\alpha,\beta+1), (\alpha+1,\beta)$, if $(i,j) >(\alpha,\beta)$, then $\varPhi[(\alpha,\beta),(i,j)]=0$, and hence $\varPhi$ is an upper triangular matrix whose entries on the main diagonal are all $\lambda-\mu$. In particular, the determinant of $\varPhi$ is $(\lambda-\mu)^{ \xi_\lambda \xi_\mu}$, and thus it is non-zero, because $\lambda \not=\mu$ by assumption. As a result, $\rank (\varPhi) = \xi_\lambda \xi_\mu$. Therefore, the desired equality follows from Proposition~\ref{prop:induction}. 
\end{proof}
Let $n \in \N$ and let $\lambda_1, \dots, \lambda_n \in \K$ be pairwise distinct.  For each $i \in \{1, \dots, n\}$, consider a Jordan matrix $J_{\lambda_i}$ with eigenvalue $\lambda_i$. Then, for each such $i$, there exist an $\ell_i \in \N$, a $(k_1^{(i)},\dots, k_{\ell_i}^{(i)}) \in \N^{\ell_i}$, and a $(r_1^{(i)}, \dots, r_{\ell_i}^{(i)}) \in \N^{\ell_i}$ with $r_1^{(i)} > \cdots > r_{\ell_i}^{(i)}$, so that $J_{\lambda_i}$ can be identified with $\bigoplus_{j=1}^{\ell_i} k_j^{(i)} J_{\lambda_i, r_j^{(i)}}$ after a suitable permutation of the Jordan blocks. For each~$i \in \{1, \dots, n\}$, let 
\begin{eqnarray}
\label{eq:jordan}
J_{\lambda_1 \cdots \lambda_i} = \bigoplus_{p=1}^i J_{\lambda_p} = \bigoplus_{p=1}^i \bigoplus_{j=1}^{\ell_p} k_j^{(p)} J_{\lambda_p, r_j^{(p)}}. 
\end{eqnarray}

Let $\xi_i = \sum_{j=1}^{\ell_i} k_j^{(i)} r_j^{(i)}$ and let $R = \K[x_1^{(i)}, \cdots, x_{\xi_i}^{(i)} \ | \ 1 \leq i \leq n]$. For each $i \in \{1, \dots, n\}$, we write $\boldsymbol x_i$ for the vector of elements from the $i$-th block of variables, $L_{\lambda_i}$ for the set of entries of $J_{\lambda_i} \boldsymbol x_i \wedge \boldsymbol x_i$, and  $I_{\lambda_i}$ for the ideal $\langle L_{\lambda_i} \rangle$. 
If $i \in \{1, \dots, n\}$, the vector of variables obtained by stacking $\boldsymbol x_1, \cdots,  \boldsymbol x_i$ vertically is denoted by $\boldsymbol x_{1\cdots i}$. Let $L_{\lambda_1 \cdots \lambda_i}$ be the set of entries of
$J_{\lambda_1 \cdots \lambda_i} 
\boldsymbol x_{1 \cdots i}
\wedge 
\boldsymbol  x_{1 \cdots i}$
and let $I_{\lambda_1 \cdots \lambda_i} = \langle L_{\lambda_1 \cdots \lambda_i} \rangle$.  
Consider the set 
\[
M = \left\{\left.x_1^{(i)}, \dots, x_{\xi_i}^{(i)} \ \right| \ 1 \leq i \leq n\right\}
\] 
of variables of $R$. For each $i \in \{1, \dots, n\}$, define $\mathfrak{p}_i$ to be the ideal generated by $M \setminus \{x_1^{(i)}, \cdots, x_{\xi_i}^{(i)}\}$.
\begin{theorem}
\label{thm:decomposition} 
Let $n \geq 2$. Then $I_{\lambda_1 \cdots \lambda_n} = \bigcap_{i=1}^n \left(\langle L_{\lambda_i}\rangle  +\mathfrak{p}_i\right)$.  
\end{theorem}
\begin{proof}
The proof is by induction on $n$.  For $n=2$, the statement
is 
just Proposition~\ref{prop:initial}. Assume now that the statement
holds for matrices with $m$ Jordan blocks and we wish to show it
for~$m+1$.
We have
$J_{\lambda_1 \cdots \lambda_{m+1}} = J_{\lambda_1 \cdots \lambda_m}
\oplus J_{\lambda_{m+1}}$,
and by the induction hypothesis, it suffices to show
\begin{equation}
\label{eq:idealDecomposition}
I_{\lambda_1 \cdots \lambda_{m+1}} = 
\left\langle L_{\lambda_1 \cdots \lambda_m}, x_1^{(m+1)}, \dots, x_{\xi_{m+1}}^{(m+1)} \right\rangle   \cap 
\left\langle L_{\lambda_{m+1}}, x_1^{(1)}, \dots, x_{\xi_1}^{(1)}, \dots,  x_1^{(m)}, \dots, x_{\xi_m}^{(m)} 
\right\rangle. 
\end{equation}
%
Define a bi-grading on $R$ by 
\[
\deg x_\alpha^{(i)} = 
\left\{
\begin{array}{ll}
(1,0) & \mbox{if $1 \leq i \leq m$ and $1 \leq \alpha \leq \xi_i$} \\
(0,1) & \mbox{if $i=m+1$ and $1 \leq \alpha \leq \xi_{m+1}$.}
\end{array}
\right. 
\]
Write $R_{(1,1)}$ for the bi-graded piece of $R$ in bi-degree $(1,1)$. Let 
\[
\Sigma = \left\{\left. (i,\alpha,\beta)\in \N^3 \ \right| \  1 \leq i \leq m,\  1 \leq \alpha \leq \xi_i,\  1 \leq \beta \leq \xi_{m+1} \right\}. 
\]
We use an ordering on $\N^2$ defined as the ordering on $\varOmega$ above. We define an ordering on $\Sigma$ by $(i,\alpha,\beta) > (j,\gamma, \delta)$ if and only if either $i >j$ or $i=j$ and $(\alpha,\beta)>(\gamma,\delta)$. 

Let $\varPhi$ be the matrix~(\ref{eq:Phi}) for $A = J_{\lambda_1 \cdots \lambda_m}$ and $B=J_{\lambda_{m+1}}$. 
For each $i \in \{1, \dots ,n\}$, let 
\[
 \Pi_{\lambda_i}^{(i)}= \bigcup_{j=1}^{\ell_i} \left\{\left.cr_{j}^{(i)} \ \right| \ 1 \leq c \leq k_j^{(i)}\right\}. 
\]
If $(i,\alpha,\beta) \in \Sigma$, then $J_{\lambda_1 \cdots \lambda_i}[\alpha,:] \boldsymbol x_{1\cdots i}= J_{\lambda_i}[\alpha,:] \boldsymbol x_{i}$. Thus, in the same way as we showed~(\ref{eq:Phigeneral}), one can show that 
\[
\varPhi[(i,\alpha,\beta),(j,\gamma,\delta)] = 
\left\{
\begin{array}{ll}
\lambda_i-\lambda_{m+1} & \mbox{if $(i,\alpha,\beta)=(j,\gamma,\delta)$,} \\
-1 & \mbox{if $i=j$, $\alpha  \in \Pi_{\lambda_i}^{(i)}$, $\beta \not\in \Pi_{\lambda_{m+1}}^{(m+1)}$, and $(\alpha, \beta+1) = (\gamma,\delta)$,} \\
-1 & \mbox{if $i=j$, $\alpha  \not\in \Pi_{\lambda_i}^{(i)}$, $\beta\not\in \Pi_{\lambda_{m+1}}^{(m+1)}$, and $(\alpha, \beta+1)= (\gamma,\delta)$,} \\
1 & \mbox{if $i=j$, $\alpha  \not\in \Pi_{\lambda_i}^{(i)}$, $\beta \in \Pi_{\lambda_{m+1}}^{(m+1)}$, and $(\alpha+1, \beta)= (\gamma,\delta)$,} \\
1 &  \mbox{if $i=j$, $\alpha  \not\in \Pi_{\lambda_i}^{(i)}$, $\beta\not\in \Pi_{\lambda_{m+1}}^{(m+1)}$, and $(\alpha+1, \beta)=(\gamma,\delta)$,} \\
0 & \mbox{otherwise.} 
\end{array}
\right. 
\]

Assume that $(j,\gamma, \delta) > (i, \alpha, \beta) $. If $i<j$, then clearly $\varPhi[(i,\alpha,\beta), (j,\gamma, \delta)] = 0$. Suppose that $i=j$. Then $ (\gamma,\delta) > (\alpha,\beta)$. As indicated above,  $\varPhi[(i,\alpha,\beta), (j,\gamma, \delta)] \not= 0$ if and only if $(j,\gamma,\delta) \in \{(i,\alpha,\beta), \ (i,\alpha,\beta+1), \ (i,\alpha+1,\beta)\}$. In each case, $(\alpha,\beta)$ is greater or equal to $(\gamma,\delta)$. Therefore, if $(\gamma,\delta) > (\alpha,\beta)$, then $\varPhi_{(i,\alpha,\beta), (j,\gamma, \delta)} = 0$. As a result, $\varPhi$ is an upper triangular matrix and the determinant of $\varPhi$ is $\prod_{i=1}^m (\lambda_i -\lambda_{m+1})^{\xi_i\xi_{m+1}}$. Since $\lambda_i \not= \lambda_{m+1}$ by assumption,  the determinant of $\varPhi$ is non-zero. This implies that $\rank (\varPhi) = (\sum_{i=1}^m \xi_i) \xi_{m+1}$. Thus (\ref{eq:idealDecomposition}) follows from Proposition~\ref{prop:induction}. Therefore, $S(m+1)$ is true under the assumption that $S(m)$ is true, and hence $S(n)$ is true for every $n \geq 2$ by induction.
\end{proof}
\begin{example}
\label{ex:diagonal}
Let $A \in \K^{r \times r}$ be a diagonalizable matrix, let $\lambda_1, \dots, \lambda_n\in K$ be the distinct eigenvalues of $A$, and let $k_i$ be the algebraic multiplicity of $\lambda_i$ for each $i \in \{1, \dots n\}$. Then $A$ is similar to a Jordan matrix $J_{\lambda_1 \cdots \lambda_n}$ of the form~(\ref{eq:jordan}) with $\ell_i=1$, $r_1^{(i)} = 1$, and $k_1^{(i)} = k_i$ for each $i \in \{1, \dots n\}$.  Since
$L_{\lambda_i} = 0$, it follows from Theorem~\ref{thm:decomposition} that 
\[
 I_{\lambda_1 \cdots \lambda_n} = \bigcap_{i=1}^n  \mathfrak{p}_i
\]
is the prime decomposition of $I_{\lambda_1 \cdots \lambda_n}$. In particular, $I_{\lambda_1 \cdots \lambda_n}$ is radical. As is stated in Proposition~\ref{prop:similar}, the two ideals $I_A$ and $I_{\lambda_1 \cdots \lambda_n}$ differ only by a linear change of coordinates, and thus $I_A$ is also radical. 
\end{example}
Let $A \in \K^{r \times r}$ and let $\lambda_1, \dots, \lambda_n \in \K$ be the distinct eigenvalues of $A$. Then $A$ is similar to a Jordan matrix $J_{\lambda_1 \cdots \lambda_n}$ of the form~(\ref{eq:jordan}),   
and thus there exists an invertible matrix~$C \in \K^{r \times r}$ such that $J_{\lambda_1 \cdots \lambda_n} = C^{-1} A C$. 
Let $Z_A$ and $Z_{\lambda_1 \cdots \lambda_n}$ be the eigenschemes of $A$ and $J_{\lambda_1 \cdots \lambda_n}$ respectively. The linear change of coordinates determined by $C$ induces an automorphism $\varphi$ of $\P^{r-1}$. The proof of Proposition~\ref{prop:similar} implies $\varphi(Z_{\lambda_1 \cdots \lambda_n}) = Z_A$. 

Let $Z_{\lambda_i}$ be the sub-scheme of  $Z_{\lambda_1 \cdots \lambda_n}$ of $I_{\lambda_i} + \mathfrak{p}_i$ for each $i \in\{1, \dots, n\}$. Then Theorem~\ref{thm:decomposition} implies that 
\[
Z_{\lambda_1 \cdots \lambda_n} = \bigcup_{i=1}^n Z_{\lambda_i}. 
\]
Let $Z_{A, \lambda_i} = \varphi(Z_{\lambda_i})$ for each $i \in \{1, \dots, n\}$. 
\begin{corollary}
\label{cor:generalized-eigenvector}
The affine cone over the scheme-theoretic linear span of $Z_{A, \lambda_i}$ coincides with the generalized eigenspace of $A$ corresponding to $\lambda_i$.
\end{corollary}
\begin{proof}
Let $i \in \{1, \dots, n\}$. By the construction of $Z_{\lambda_i}$ and Corollary~\ref{cor:nondegenerate}, the scheme-theoretic linear span of $Z_{\lambda_i}$ is the linear subspace defined by $\mathfrak{p}_i$ and coincides with the projectivization of the generalized eigenspace of $J_{\lambda_1 \cdots \lambda_n}$ corresponding to $\lambda_i$. 

Corollary~\ref{cor:generalized-eigenvector} follows from the fact that $\boldsymbol v_i$ is a generalized eigenvector of $J_{\lambda_1 \cdots \lambda_n}$ corresponding to $\lambda_i$ if and only if $C^{-1} \boldsymbol v_i$ is a generalized eigenvector of $A$ corresponding to~$\lambda_i$. 
\end{proof}
For each $i \in \{1, \dots, n\}$, let
$R_i = \K[x_1^{(i)}, \dots, x_{\xi_i}^{(i)}]$, let
$I_{\lambda_i}$ be the ideal of $R_i$ generated by $L_{\lambda_i}$,
and let
$I_{\lambda_i} = \bigcap_{j=1}^{\ell_i} \mathfrak{q}_{\lambda_i,j}$ be
the primary decomposition of $I_{\lambda_i}$ given in
Proposition~\ref{prop:decomposition}.  By Theorem~\ref{thm:primary},
Proposition~\ref{prop:hilbertFunction}, and
Theorem~\ref{thm:decomposition}, one obtains the following corollary.
\begin{corollary}
\label{cor:primary-decomposition}
The ideal $I_{\lambda_1 \cdots \lambda_n}$ of $J_{\lambda_1 \cdots \lambda_n}$ has an irredundant primary decomposition 
\[
I_{\lambda_1 \cdots \lambda_n} = \bigcap_{i=1}^n \left[\bigcap_{j=1}^{\ell_i} \left(\left\langle \mathfrak{q}_{\lambda_i,j} \right\rangle+\mathfrak{p}_i\right)\right]. 
\] 
Furthermore, for each $i \in \{1, \dots, n\}$ and $j \in \{1, \dots, \ell_i\}$, 
\[
H_{R/\left(\left\langle \mathfrak{q}_{\lambda_i,j} \right\rangle+\mathfrak{p}_i\right)}(t) =  H_{R_i/\mathfrak{q}_{\lambda_i,j}}(t) =  r_j^{(i)}\, {t+k_1^{(i)}+ \cdots k_j^{(i)}-1 \choose t}.
\]
\end{corollary}
From Example~\ref{ex:diagonal}, and Corollaries~\ref{cor:diagonalizability} and~\ref{cor:primary-decomposition} one obtains the following corollary.    
\begin{corollary}
\label{cor:diagonalizability-general}
Let $A \in \K^{r \times r}$ have eigenvalues in~$\K$. The following
conditions are equivalent:
\begin{itemize}
\item[(i)] The matrix $A$ is diagonalizable.  
\item[(ii)] The ideal $I_A$ of $A$ is radical. 
\item[(iii)] The eigenscheme $Z_A$ of $A$ is reduced. 
\end{itemize}
\end{corollary}
The following is a consequence of Proposition~\ref{prop:curvilinear}
and Theorem~\ref{thm:decomposition}.
\begin{corollary}
Let $A \in \K^{r \times r}$. If the eigenscheme $Z_A$ of $A$ is of dimension $0$, then $Z_A$ is curvilinear. 
\end{corollary} 

\section{Eigenschemes and tangent bundles} 
\label{sec:tangent-bundle}
The rest of the paper concerns an interpretation of eigenschemes via tangent bundles.
The eigenscheme of an $r \times r$ matrix can be expressed as the zero scheme of a global section of the tangent bundle on $\P^{r-1}$.  This idea has been known to experts for a long time and also appears in~\cite{ottaviani-oeding}.  We expose it here because of its natural connection to the material presented.
We use standard notation from algebraic geometry and assume
familiarity with at least Chapters I and II in~\cite{hartshorne}.
In the remaining two sections, the field $\K$ is assumed to be algebraically closed of characteristic~$0$.  This assumption is inherited from~\cite[Lemma~2.5]{ein}.
%

Let $K^\bullet$ be the Koszul complex, i.e., the complex $\{K^\ell, \ \wedge^\ell \boldsymbol x\}_{0 \leq \ell \leq r-1}$ with 
\[
 K^\ell = \bigwedge^\ell \K^r \otimes \mathcal{O}_{\P^{r-1}}(\ell)
\]
and  $ \wedge^\ell \boldsymbol x: \bigwedge^\ell \K^r \otimes \mathcal{O}_{\P^{r-1}}(\ell) \rightarrow \bigwedge^{\ell+1} \K^r  \otimes \mathcal{O}_{\P^{r-1}}(\ell+1)$ given by 
\[
  \wedge^\ell \boldsymbol x \   (\mathbf{e}_{i_1} \wedge \cdots \wedge \mathbf{e}_{i_\ell}) = \sum_{i=0}^{r-1} x_i\ \mathbf{e}_i \wedge \mathbf{e}_{i_1} \wedge \cdots \wedge \mathbf{e}_{i_\ell}. 
\]
The bundle $T_{\mathbb{P}^{r-1}}$ is the image of $\wedge^1 \boldsymbol x: \K^r \otimes \mathcal{O}_{\P^{r-1}}(1) \rightarrow \bigwedge^2 \K^r \otimes \mathcal{O}_{\P^{r-1}}(2)$: 
\begin{equation*}
\xymatrix{
0 \ar[r]  & \mathcal{O}_{\P^{r-1}} \ar[r]  & \K^r \otimes \mathcal{O}_{\P^{r-1}}(1) \ar[r]^{\hspace{5mm} \wedge^1 \boldsymbol x} &  T_{\P^{r-1}} \ar[r] & 0. 
}
\end{equation*}
Note that $H^0(\P^{r-1}, \mathcal{O}_{\P^{r-1}}) \simeq \K$, $H^1(\P^{r-1}, \mathcal{O}_{\P^{r-1}}) = 0$, and  $H^0(\P^{r-1}, \K^r \otimes \mathcal{O}_{\P^{r-1}}(1)) \simeq \K^{r \times r}$. Thus, taking cohomology yields the following exact sequence: 
\[
\xymatrix{
0 \ar[r] & \K \ar[r]  &\K^{r \times r} \ar[rr]^{H^0(\wedge^1\boldsymbol x) \hspace{5mm}}   & & H^0 (\P^{r-1}, T_{\P^{r-1}}) \ar[r]  & 0.
}
\]
Since the image of $A\boldsymbol x$ under $H^0(\wedge^1\boldsymbol x)$
equals $A\boldsymbol x \wedge \boldsymbol x$, we obtain the set
equality
\begin{eqnarray*}
  H^0(\P^{r-1}, T_{\P^{r-1}}) 
  & = & \{ A\boldsymbol x \wedge\boldsymbol x\ | \ A \in \K^{r\times r}\}.   
\end{eqnarray*}
This shows that every $r \times r$ matrix $A$ yields a global section
$s_A$ of $T_{\P^{r-1}}$.  The kernel of the map consists of
scalar multiples of the identity matrix. 
An eigenvector $\boldsymbol v \in \K^r$ of $A$ corresponds to a point
of the zero scheme $(s_A)_0$ of $s_A$. Therefore, the definition of an
eigenvector of a linear operator can be rephrased as follows.
A non-zero vector $\boldsymbol v \in \K^r$ is an eigenvector of $A$ if
$[\boldsymbol v]$ lies in the zero scheme of the global section $s_A$
of $T_{\P^{r-1}}$.
\begin{remark}
Let $A \in \K^{r \times r}$ be generic. Then the global section $s_A$ of $T_{\P^{r-1}}$ is \emph{regular}: its zero locus has pure codimension~$r-1$. Its zero scheme is therefore zero-dimensional and its length is the ${(r-1)}$-st Chern class of~$T_{\P^{r-1}}$. In other words, the number of linearly independent eigenvectors for $A$ coincides with $c_{r-1}(T_{\P^{r-1}}) = r$, as expected.
\end{remark}
\begin{remark}
In Corollary~\ref{cor:diagonalizability-general}, we showed that the condition for a matrix $A \in \K^{r \times r}$  to be diagonalizable is equivalent to the condition for the eigenscheme $Z_A$ of $A$ to be reduced. Since $Z_A$ can be identified with the zero scheme $(s_A)_0$ of the global section $s_A$ of $T_{\P^{r-1}}$,
the matrix $A$ is diagonalizable if and only if $(s_A)_0$ is reduced.  
\end{remark}
Let $s_0, s_1 \in H^0(\P^{r-1}, T_{\P^{r-1}})$ be generic and let $E$
be their dependency locus, i.e.,
\[
E = \left\{\left. [\boldsymbol v]  \in \P^{r-1} \ \right| \ \mbox{$s_0([\boldsymbol v])$ and $s_1([\boldsymbol v])$ are linearly dependent} \right\}. 
\]
Lemma~2.5 in~\cite{ein}, which requires $\K$ to be algebraically closed of characteristic $0$, implies that $E$ is a non-singular curve.
\begin{proposition}
\label{prop:genus} The genus of $E$ is $(r-1)(r-2)/2$. 
\end{proposition}
\begin{proof}
Let $\boldsymbol s = (s_0,s_1)$ be the sheaf morphism from the direct sum  $2\cO_{\P^{r-1}}$ of two copies of~$\cO_{\P^{r-1}}$  to $T_{\P^{r-1}}$. The Eagon-Northcott complex of the homomorphism of vector bundles~$\boldsymbol s^\vee : \Omega^1_{\P^{r-1}} \rightarrow 2\cO_{\P^{r-1}}$ is:
\begin{equation}
\label{eq:eagon-northcott}
0 \rightarrow (r-2)\Omega_{\P^{r-1}}^{r-1} \rightarrow (r-3) \Omega_{\P^{r-1}}^{r-2} \rightarrow \cdots \rightarrow 2 \Omega_{\P^{r-1}}^3 \rightarrow \Omega_{\P^{r-1}}^2\stackrel{\wedge^2 \boldsymbol s^\vee}{\longrightarrow} \cO_{\P^{r-1}}, 
\end{equation}
where $(i-1)\Omega_{\P^{r-1}}^i$ denotes the direct sum of $(i-1)$ copies of the $i$-th exterior power of~$\Omega_{\P^{r-1}}^1$.
Since $E$ has the expected codimension, complex~(\ref{eq:eagon-northcott}) is a locally free resolution of~$\cO_E$. Therefore, we obtain 
\[
\chi (\cO_E) = 1 - \sum_{j=2}^{r-1} (-1)^j(j-1)\chi\left(\Omega_{\P^{r-1}}^j\right). 
\]
Thus the Riemann-Roch theorem implies that 
\[
g(E) = \sum_{j=2}^{r-1} (-1)^j(j-1)\chi\left(\Omega_{\P^{r-1}}^j\right). 
\]
By the Bott formula for $\P^{r-1}$ (see, for example, \cite{okonek-schneider-spindler}), we obtain 
\begin{eqnarray*}
\chi \left(\Omega_{\P^{r-1}}^j\right) & = & \sum_{k=0}^{r-1} (-1)^k \dim H^k\left(\P^{r-1}, \Omega_{\P^{r-1}}^j\right) \\
& = & (-1)^j \dim H^j\left(\P^{r-1}, \Omega_{\P^{r-1}}^j\right) \\
& = & (-1)^j. 
\end{eqnarray*} 
Therefore, 
\[
g(E) =  \sum_{j=2}^{r-1} (j-1) = \frac{(r-1)(r-2)}{2}. \qedhere
\]
\end{proof} 

\section{The discriminant}
\label{sec:discriminant}
The \emph{discriminant hypersurface} is the hypersurface in
$\P(\K^{r \times r})$ formed by $r \times r$ matrices with entries
from~$\K$ which are not diagonalizable.  It can also be thought of as
the vanishing hypersurface of the discriminant of the characteristic
polynomial of a generic matrix. The degree of this hypersurface
is~$r(r-1)$ as can be checked by comparing the discriminant to the
resultant of the characteristic polynomial and its derivative.  Using
our results it is possible to express this degree geometrically
in terms of the $(r-2)$-nd Chern class of the tangent bundle
on~$\P^{r-1}$.

We define the \emph{discriminant}
$\mathcal{D} \subseteq \P(\K^{r\times r})$ to be the Zariski closure
of the set
\[
\left\{\left. [A]  \in \P(\K^{r\times r}) \ \right| \ \mbox{$A$ is non-diagonalizable}  \right\}.
\] 
\begin{theorem}
The discriminant $\mathcal{D}$ is a hypersurface of degree $2
c_{r-2}(T_{\P^{r-1}})$.
\end{theorem}
\begin{proof}
First note that $c_{r-2}(T_{\P^{r-1}})=r(r-1)/2$ (see Example 3.2.11
in \cite{fulton}). We need to show that the intersection of
$\mathcal{D}$ with a generic line in $\P(\K^{r\times r})$ consists of
$r(r-1)$ points. We do this by considering a pencil
\[
L = \left\{[ \lambda B+ \mu C] \ \left| \  [\lambda : \mu] \in \P^1\right. \right\}
\]
formed by two generic $r \times r$ matrices $B$ and $C$ and counting
the number of elements in $L$ that correspond to non-diagonalizable
matrices.

Given a matrix $A$, recall that $Z_A$ denotes the corresponding eigenscheme. Let $\varSigma$ be the incidence correspondence: 
\[
\varSigma = \left\{\left. \left([A], [\boldsymbol v]\right) \in \P(\K^{r \times r}) \times \P^{r-1} \ \right| \ \mbox{$[\boldsymbol v] \in Z_A$}  \ \right\}, 
\]
and let $\pi_1$ and $\pi_2$ be the projections from $\varSigma$ to $\P(\K^{r \times r})$ and $\P^{r-1}$ respectively. Consider the following subset of $\varSigma$: 
\[
\varSigma_0 = \left\{\left. \left([A], [\boldsymbol v]\right) \in \P(\K^{r \times r}) \times \P^{r-1} \ \right| \ \mbox{$Z_A$ is singular at $[\boldsymbol v]$}  \ \right\}. 
\] 
Corollary~\ref{cor:diagonalizability-general} implies that $A$ is not diagonalizable if and only if $Z_A$ is singular. This implies that $\mathcal{D}$ can be identified with the Zariski closure of ${\pi_1(\varSigma_0)}$. 

Let $E = \pi_2 (\pi_1^{-1} (L))$. Then $E$ is a curve, which is obtained as the dependency locus of two global sections of $T_{\P^{r-1}}$. 
As was mentioned before, since $s_B$ and $s_C$ are generic because of the choices of $B$ and $C$, Lemma~2.5 in~\cite{ein} implies that $E$ is non-singular of codimension $r-2$ or $E$ is a non-singular curve. Furthermore, the degree of $E$ is $c_{r-2}(T_{\P^{r-1}}) = r(r-1)/2$ (see, for example, Example 14.4.1 in \cite{fulton}). 

The pencil of the divisors $(s_B)_0$ and $(s_C)_0$ of $E$ defines an $r:1$ morphism $\varPhi$ from $E$ to~$L$. To find the number of elements of $L$ that correspond to non-diagonalizable matrices, it is therefore equivalent to finding the number of branch points of $\varPhi$. 

Let $P$ be the ramification divisor of $\varPhi$. Then it follows from the Riemann-Hurwitz formula (see, for example, Corollary~2.4 in~\cite[Chapter~4]{hartshorne}) that 
\[
2g(E) - 2 = r(2g(L)-2) + \deg P =\deg P  -2r,  
\] 
where $g(E)$ and $g(L)$ are genera of $E$ and $L$ respectively. Notice that, because of the genericity of $B$ and $C$, the ramification index for a point of $E$ is at most two. Therefore, the number of branch points of $\varPhi$ equals $\deg P = 2g(E)+2(r-1)$. Thus, by Proposition~\ref{prop:genus}, we get
\[
\deg P = 2g(E)+2(r-1) = (r-1)(r-2)+2(r-1) = r(r-1) = 2c_{r-2}(T_{\P^{r-1}}), 
\]
which completes the proof. 
\end{proof}
\begin{remark}
In~\cite{ASS}, the authors used similar ideas to generalize Theorem~6.1 to tensors, where the formula for the degree of the discriminant of a tensor (i.e., the closure of the locus of tensors that have fewer eigenvectors than expected) is given (see~\cite[Corollary~4.2]{ASS}).  
\end{remark}




\begin{thebibliography}{aaa}
%
\bibitem{ASS}
H.~Abo, A.~Seigal, and B.~Sturmfels, 
{\it Eigenconfigurations of tensors}, 
preprint, 
{\tt arXiv:1505.05729}. 
%
\bibitem{catalano-johnson}
M.~L.~Catalano-Johnson, 
{\it The resolution of the ideal of $2\times 2$ minors of a $2\times n$ matrix of linear forms}, 
 J. Algebra {\bf 187} (1997), no. 1, 39--48
%
\bibitem{ein}
L.~Ein, 
{\it Some stable vector bundles on $\P^4$ and $\P^5$}, 
  J.~Reine Angew.~Math. {\bf 337} (1982), 142--153.
%
\bibitem{eisenbud-sturmfels}
D.~Eisenbud and B.~Sturmfels, 
{\it Binomial ideals}, 
Duke Math. J. {\bf 84} (1996), no. 1, 1--45.
%
\bibitem{fulton}
W.~Fulton, 
{\it Intersection theory. Second edition}, 
Ergebnisse der Mathematik und ihrer Grenzgebiete. 3. Folge. A Series
of Modern Surveys in Mathematics, 2. Springer-Verlag, Berlin,
1998. xiv+470 pp.
%
\bibitem{M2}
D.~Grayson and M.~Stillman,
{\it Macaulay2, a software system for research 
                   in algebraic geometry},
Available at  \href{http://www.math.uiuc.edu/Macaulay2/}
                   {http://www.math.uiuc.edu/Macaulay2/}
%
\bibitem{harris}
J.~Harris, 
{\it Algebraic geometry, A first course.} Graduate Texts in Mathematics, 133. Springer-Verlag, New York, 1992. xx+328 pp. 
%
\bibitem{hartshorne}
R.~Hartshorne, 
{\it Algebraic geometry}, 
 Graduate Texts in Mathematics, No. 52. Springer-Verlag, New York-Heidelberg, 1977. xvi+496 pp. 
 %
\bibitem{kahle2012}
T.~Kahle, 
{\it Decompositions of binomial ideals}, 
 J. Softw. Algebra Geom. {\bf 4} (2012), 1--5.
 %
\bibitem{kahle2010} 
\bysame, 
{\it Decompositions of binomial ideals}, 
Ann. Inst. Statist. Math {\bf 62} (2010), no. 4, 727--745. 
%
\bibitem{KM2014}
T.~Kahle and E.~Miller,
{\it Decompositions of commutative monoid congruences and binomial ideals},
Algebra \& Number Theory {\bf 8} (2014), no.~6, 1297--1364.
%
\bibitem{lim}
L.~H.~Lim, {\it Singular values and eigenvalues of tensors: a variational approach}, 
Proceedings of the IEEE International Workshop on Computational Advances in Multi-Sensor Adaptive Processing (CAMSAP '05), 1 (2005), pp. 129--132. 
%
\bibitem{castilla-sanchez}
I.~O.~ Mart\'{\i}nez de Castilla and R.~P.~S\'anchez, 
{\it Cellular binomials. Primary decomposition of binomial ideals}, 
 J. Symbolic Comput. {\bf 30} (2000), no. 4, 383--400. 
%
\bibitem{nguyen-thieu-vu}
H.~D.~Nguyen, P.~D.~Thieu, and T.~Vu, 
{\it Koszul determinantal rings and $2 \times e$ matrices of linear forms}, 
to appear in Michigan Math. J., {\tt arXiv:1309.4698v2}. 
%
\bibitem{ottaviani-oeding}
L.~Oeding and G.~Ottaviani, 
{\it Eigenvectors of tensors and algorithms for Waring decomposition}, 
J. Symbolic Comput. {\bf 54} (2013), 9--35.
%
\bibitem{ORS}
L.~Oeding, E.~Robeva and B.~Sturmfels, 
Decomposing tensors into frames, preprint, {\tt arXiv:1504.08049}.
%
\bibitem{qi}
L.~Qi, {\it Eigenvalues of a real supersymmetric tensor}, 
 J. Symbolic Comput. {\bf 40} (2005), no. 6, 1302--1324.
 %
\bibitem{okonek-schneider-spindler}
C.~Okonek, M.~Schneider, and H.~Spindler, 
{\it Vector bundles on complex projective spaces. Corrected reprint of the 1988 edition. With an appendix by S. I. Gelfand}, 
 Modern Birkh\"auser Classics. Birkh\"auser/Springer Basel AG, Basel, 2011. viii+239 pp. 
 %
 \bibitem{Rob} E.~Robeva, 
{\em Orthogonal decomposition of symmetric tensors}, 
SIAM J.~Matrix Anal. Appl. {\bf 37} (1) (2016) 86--102. 
%
\bibitem{zaare-nahandi}
 R.~Zaare-Nahandi and R.~Zaare-Nahandi, 
 {\it  Gr\"obner basis and free resolution of the ideal of $2$-minors of a $2 \times n$ matrix of linear forms},
  Comm. Algebra {\bf 28} (2000), no. 9, 4433--4453.
\end{thebibliography}
\end{document}